\documentclass{amsart}
\usepackage{amsmath, amsthm, amssymb, amsfonts} 
\usepackage{enumerate}
\usepackage{verbatim}
\usepackage{esint}
\usepackage{color}
\usepackage{mathtools}
\usepackage{xcolor}
\usepackage[T1]{fontenc}

\usepackage{tikz,amsthm,amsmath,amstext,amssymb,amscd,epsfig,euscript, mathrsfs,
multicol,graphpap,graphics,graphicx,times,enumerate,
wrapfig,color,pict2e}
\usepackage{setspace}
\usepackage{caption}
\usepackage{subcaption}
\usepackage{hyperref}

\calclayout

\numberwithin{equation}{section}

\title{Infinitesimal splitting for spaces with thick curve families and Euclidean embeddings} 
\date{\today}

\author{Guy C. David}
\address{Department of Mathematical Sciences\\ Ball State University, Muncie, IN 47306}
\email{gcdavid@bsu.edu}

\author{Sylvester Eriksson-Bique}
\address{Department of Mathematics and Statistics, University of Jyv\"askyl\"a, P.O. Box 35 (MaD), FI-40014, Jyv\"askyl\"a, Finland}
\email{sylvester.d.eriksson-bique@jyu.edu}


\thanks{G.~ C.~ David was partially supported by the National Science Foundation under Grant No. DMS-1758709. S.~ Eriksson-Bique was partially supported by the National Science Foundation under Grant No. DMS-1704215. Eriksson-Bique is also thankful for IMPAN for hosting the semester ``Geometry and analysis in function and mapping theory on Euclidean and metric measure space" where part of this research was conducted. This work was also partially supported by the grant \#346300 for IMPAN from the Simons Foundation and the matching 2015-2019 Polish MNiSW fund.}

\begin{document}
\theoremstyle{plain}
\newtheorem{theorem}{Theorem}
\newtheorem{exercise}{Exercise}
\newtheorem{corollary}[theorem]{Corollary}
\newtheorem{scholium}[theorem]{Scholium}
\newtheorem{claim}[theorem]{Claim}
\newtheorem{lemma}[theorem]{Lemma}
\newtheorem{sublemma}[theorem]{Lemma}
\newtheorem{proposition}[theorem]{Proposition}
\newtheorem{conjecture}{theorem}
\newtheorem{maintheorem}{Theorem}
\newtheorem{maincor}[maintheorem]{Corollary}
\renewcommand{\themaintheorem}{\Alph{maintheorem}}

\theoremstyle{definition}
\newtheorem{fact}[theorem]{Fact}
\newtheorem{example}[theorem]{Example}
\newtheorem{definition}[theorem]{Definition}
\newtheorem{remark}[theorem]{Remark}
\newtheorem{question}[theorem]{Question}

\numberwithin{equation}{section}
\numberwithin{theorem}{section}

\newcommand{\cG}{\mathcal{G}}
\newcommand{\RR}{\mathbb{R}}
\newcommand{\HH}{\mathcal{H}}
\newcommand{\LIP}{\textnormal{LIP}}
\newcommand{\Lip}{\textnormal{Lip}}
\newcommand{\Tan}{\textnormal{Tan}}
\newcommand{\WTan}{\textnormal{WTan}}
\newcommand{\length}{\textnormal{length}}
\newcommand{\dist}{\textnormal{dist}}
\newcommand{\diam}{\textnormal{diam}}
\newcommand{\vol}{\textnormal{vol}}
\newcommand{\rad}{\textnormal{rad}}
\newcommand{\side}{\textnormal{side}}

\def\bA{{\mathbb{A}}}
\def\bB{{\mathbb{B}}}
\def\bC{{\mathbb{C}}}
\def\bD{{\mathbb{D}}}
\def\bR{{\mathbb{R}}}
\def\bS{{\mathbb{S}}}
\def\bO{{\mathbb{O}}}
\def\bE{{\mathbb{E}}}
\def\bF{{\mathbb{F}}}
\def\bH{{\mathbb{H}}}
\def\bI{{\mathbb{I}}}
\def\bT{{\mathbb{T}}}
\def\bZ{{\mathbb{Z}}}
\def\bX{{\mathbb{X}}}
\def\bP{{\mathbb{P}}}
\def\bN{{\mathbb{N}}}
\def\bQ{{\mathbb{Q}}}
\def\bK{{\mathbb{K}}}
\def\bG{{\mathbb{G}}}

\def\nrj{{\mathcal{E}}}
\def\cA{{\mathscr{A}}}
\def\cB{{\mathscr{B}}}
\def\cC{{\mathscr{C}}}
\def\cD{{\mathscr{D}}}
\def\cE{{\mathscr{E}}}
\def\cF{{\mathscr{F}}}
\def\cB{{\mathscr{G}}}
\def\cH{{\mathscr{H}}}
\def\cI{{\mathscr{I}}}
\def\cJ{{\mathscr{J}}}
\def\cK{{\mathscr{K}}}
\def\cL{{\mathscr{L}}}
\def\Layer{{\rm Layer}}
\def\cM{{\mathscr{M}}}
\def\cN{{\mathscr{N}}}
\def\cO{{\mathscr{O}}}
\def\cP{{\mathscr{P}}}
\def\cQ{{\mathscr{Q}}}
\def\cR{{\mathscr{R}}}
\def\cS{{\mathscr{S}}}
\def\Up{{\rm Up}}
\def\cU{{\mathscr{U}}}
\def\cV{{\mathscr{V}}}
\def\cW{{\mathscr{W}}}
\def\cX{{\mathscr{X}}}
\def\cY{{\mathscr{Y}}}
\def\cZ{{\mathscr{Z}}}

\def\co{\colon}

  \def\del{\partial}
  \def\diam{{\rm diam}}
  \def\supp{{\rm supp}}
  
    \def\Curv{{\rm Curv}}
    \def\dom{{\rm dom}}
		  \def\im{{\rm im}}
			 \def\len{{\rm len}}
    \def\Geo{{\rm Geo}}
    
      \def\Frag{{\rm Frag}}
			    \def\Lines{{\rm Lines}}
			  \def\Cone{{\rm Cone}}
				
	\def\VV{{\mathcal{V}}}
	\def\FF{{\mathcal{F}}}
	\def\QQ{{\mathcal{Q}}}
	\def\BB{{\mathcal{B}}}
	\def\XX{{\mathcal{X}}}
	\def\PP{{\mathcal{P}}}
	\def\M{{\mathbb{M}}}

\def\image{{\rm Image}}
	\def\domain{{\rm Domain}}
	\def\Mod{{\rm Mod}}
	\def\Tan{{\rm Tan}}
  \def\dist{{\rm dist}}
	\def\dim{{\rm dim}}
	\def\cdim{{\rm cdim_{AR}}}
	\newcommand{\Gr}{\mathbf{Gr}}
\newcommand{\md}{\textnormal{md}}
\newcommand{\Haus}{\textnormal{Haus}}

\begin{abstract}
We study metric measure spaces that admit ``thick'' families of rectifiable curves or curve fragments, in the form of Alberti representations or curve families of positive modulus. We show that such spaces cannot be bi-Lipschitz embedded into any Euclidean space unless they admit some ``infinitesimal splitting'': their tangent spaces are bi-Lipschitz equivalent to product spaces of the form $Z\times \RR^k$ for some $k\geq 1$. We also provide applications to conformal dimension and give new proofs of some previously known non-embedding results.
\end{abstract}
\maketitle

\section{Introduction}

Many natural problems in analysis on metric spaces involve studying spaces that support ``thick'' families of rectifiable curves (or curve fragments), in one sense or another. In this paper, we show that such spaces cannot be bi-Lipschitz embedded into any Euclidean space \textit{unless} they admit some ``infinitesimal splitting'': their tangent spaces are bi-Lipschitz equivalent to product spaces of the form $Z\times \RR^k$ for some $k\geq 1$. Our methods are relatively direct and admit a number of consequences.

\subsection{Background}
In 1999, Cheeger \cite{Ch} proved a deep extension of Rademacher's theorem (Lipschitz functions are differentiable almost everywhere) to certain abstract metric measure spaces. These are the so-called \textit{PI spaces}, those that are doubling and support a Poincar\'e inequality in the sense of \cite{HK}. As one of many consequences, he showed that if a PI space admits a bi-Lipschitz embedding into some Euclidean space, then its tangent spaces must be bi-Lipschitz equivalent to Euclidean spaces almost everywhere: 
\begin{theorem}[Cheeger \cite{Ch}, Theorem 14.1]\label{thm:cheeger}
Let $(X,\mu)$ be a PI space. Suppose that $X$ admits a bi-Lipschitz embedding into some $\RR^n$. Then for $\mu$-almost-every $x\in X$, there is an integer $k\geq 1$ such that every tangent space $(Y,y)\in\Tan(X,x)$ is bi-Lipschitz equivalent to $\RR^k$.
\end{theorem}
(Here, the notion of ``tangent'' is in the pointed Gromov-Hausdorff sense; see section \ref{subsec:tangents}.)

In other words, to admit a bi-Lipschitz embedding into a Euclidean space, a PI space must itself be infinitesimally Euclidean. (Extensions of this are known, see \cite[Theorem 14.2]{Ch} and more recent results in \cite{CK_Banach, CK_Banach2, Schioppa, GCD, CKS, DK}.)

Since we know of many abstract PI spaces that are \textit{not} infinitesimally Euclidean, this consequence of Cheeger's result can be viewed as a generalized non-bi-Lipschitz embedding theorem, i.e., a checkable criterion for a space to admit no bi-Lipschitz embedding into any Euclidean space.

A Poincar\'e inequality is sufficient but not necessary to prove a Rademacher-type theorem in metric spaces, and hence a non-embeddability criterion like Theorem \ref{thm:cheeger}. Indeed, a number of weaker sufficient conditions implying Cheeger's Rademacher theorem have been found since its discovery \cite{KeithLD, Bate, Schioppa}. Most importantly for our purposes, Bate \cite{Bate}, building on work of Alberti \cite{Alberti} and Alberti-Cs\"ornyei-Preiss \cite{ACP}, showed that Cheeger's differentiable structure is equivalent to the presence of a \textit{universal family} of \textit{Alberti representations}. An Alberti representation of a measure is a decomposition into $1$-rectifiable measures supported on fragments of rectifiable curves (see Definition \ref{def:alberti}). For example, Fubini's theorem gives simple Alberti representations of Lebesgue measure on $[0,1]^2$.

Essentially, Bate shows that if a space supports a ``large enough'' family of independent Albert representations, then it supports a Rademacher theorem for Lipschitz functions, from which one can deduce an analog of the non-embeddability criterion Theorem \ref{thm:cheeger}.

Given the above, it is natural to ask whether there are conditions, weaker than any of those studied above, that are not strong enough to yield a Rademacher-type theorem but that still prevent bi-Lipschitz embeddings.

In this paper, we answer this question by studying spaces that support a single Alberti representation (or more generally $k$ independent independent Alberti representations), but not necessarily enough to form a ``universal'' family in Bate's sense, and thus not necessarily enough to yield a differentiable structure for Lipschitz functions.

Nonetheless, we show that such smaller families of Alberti representations still strongly constrain the ability of the space to bi-Lipschitz embed into any Euclidean space. The following is our main theorem.

\begin{theorem}\label{thm:mainthm}
Let $X\subseteq \RR^n$ be a closed set supporting a doubling Radon measure $\mu_0$. Suppose that a non-trivial Radon measure $\mu\ll\mu_0$ supports $k$ independent Alberti representations, for some $k\geq 1$.

Then for $\mu$-almost-every $x\in X$, there is a $k$-dimensional vector subspace $V\subseteq \RR^n$ with the following property: 

Every intrinsic tangent $Y\in\Tan_{\RR^n}(X,x)$ of $X$ at $x$ is a product $Z\times V$, for some closed set $Z\subseteq V^{\bot}\subseteq \RR^n$.
\end{theorem}
For a vector subspace $V \subset \RR^n$, its orthogonal complement is denoted by $V^{\bot}. $ An ``intrinsic tangent'' of a subset $X\subseteq\RR^n$ is simply a limit of rescalings of $X$ centered at a fixed basepoint. For a precise definition, see section \ref{subsec:tangents} below.

It is easy to recast Theorem \ref{thm:mainthm} as a result that constrains bi-Lipschitz embeddings of metric spaces:

\begin{corollary}\label{cor:bilip}
Let $X$ be a complete metric space supporting a doubling Radon measure $\mu_0$. Suppose that a non-trivial Radon measure $\mu\ll\mu_0$ supports $k$ independent Alberti representations, for some $k\geq 1$.

If $X$ admits a bi-Lipschitz embedding into some Euclidean space, then for $\mu$-almost-every $x\in X$ and every tangent $(Y,y)\in\Tan(X,x)$, $Y$ is bi-Lipschitz equivalent to a product $Z\times \RR^k$, for some complete metric space $Z$.
\end{corollary}

Corollary \ref{cor:bilip} gives a simple non-bi-Lipschitz embedding criterion that applies to a wider class of examples than Theorem \ref{thm:cheeger}. 

Theorem \ref{thm:mainthm} and Corollary \ref{cor:bilip} are proven without recourse to any more general metric measure Rademacher theorem. Rather, their proofs rely only on one of the preliminary results of Bate's paper \cite{Bate} (that Alberti representations induce ``partial derivatives'' almost everywhere) -- see Proposition \ref{prop:bate} -- and an adaptation of a principle of Preiss \cite{Preiss} about the structure of the space of tangent objects, Proposition \ref{prop:basepoint}.

\begin{remark}
We note that there are many other conditions, independent from any of those discussed above, that prevent or constrain bi-Lipschitz embeddability. In particular, the results of \cite{La00} and \cite{OS14} rest also, remarkably, on studying a single family of curves that is ``thick'' in some quantitative sense (although different than the senses used here). This is part of a larger program to characterize metric spaces embedding into the so-called RNP Banach spaces, and conversely to characterize RNP-Banach spaces by the spaces embedding into them; see \cite{OSRN}.
\end{remark}

\subsection{Corollaries of the main results}\label{subsec:corollaries}

We view Corollary \ref{cor:bilip} as a  generalized non-embedding result. Since ``thick" families of curves arise in many settings, it has a number of specific consequences.

\subsubsection{Modulus}
From our perspective, the most important consequences of Corollary \ref{cor:bilip} involve its relationship with a well-known way of measuring the ``thickness'' of a family of curves $\Gamma$ in a metric measure space $(X,\mu)$. This is the \textit{$p$-modulus} of the family ($p\geq 1$), which we denote $\Mod_p(\Gamma,\mu)$. This notion plays a central role in the modern theory of analysis on metric spaces \cite{He,HKST}. (We give a precise definition in Section \ref{sec:modulus}.)

In Proposition \ref{prop:modmeasure} and Corollary \ref{cor:modmeasure} below, we show that path families of positive modulus induce non-trivial Alberti representations. This idea is essentially already contained in work of the second-named author and his collaborators \cite{Durand} and the proof below reworks that argument in a slightly different context. This is also closely related to the results of \cite{Amb} and \cite{exnerova2019plans}. The formulation in Proposition \ref{prop:modmeasure} is slightly different and applies more directly to our setting. (To link modulus and Alberti representations, we also use ideas of Keith \cite{Keith} and Bate \cite[Corollary 5.8]{Bate}.)

As a consequence of Proposition \ref{prop:modmeasure} and Corollary \ref{cor:bilip}, we obtain the following result for spaces that contain curve families of positive modulus.

\begin{corollary}\label{cor:positivemodulus}
Let $X$ be a complete metric space admitting a Radon measure $\mu$ that is absolutely continuous with respect to a doubling measure $\mu_0$. Suppose that $X$ contains family $\Gamma\subseteq\Curv(X)$ of non-constant curves with $\Mod_p(\Gamma,\mu)>0$ for some $p \in [1,\infty)$.

If $X$ admits a bi-Lipschitz embedding into some Euclidean space, then there exists a non-trivial Radon measure $\mu'\ll\mu$, such that, for $\mu'$-almost every $x \in X$, every tangent $Y$ of $X$  at $x$ is bi-Lipschitz to a product $Z \times \bR$, for some complete metric space $Z$.
\end{corollary}

For many metric measure spaces, it is known, or not difficult to check, that they contain a path family of positive modulus and do not have tangents that split as $Z\times\RR$. Therefore such metric spaces cannot bi-Lipschitz embed into any Euclidean space. We give some applications of this argument below.

\subsubsection{Conformal dimension}
An important problem in metric geometry is to understand the \textit{conformal dimension} of a metric space, a quasisymmetric invariant first introduced by Pansu \cite{P89}, and much used since \cite{MT}. There are a number of variations of this quantity, but we focus on the \textit{Ahlfors regular conformal dimension}. This is a variant first named by Bonk-Kleiner in \cite{BK04}, where they attribute the idea to Bourdon-Pajot \cite{BourdonPajot}.

We recall that a metric space $X$ is \textit{Ahlfors $Q$-regular} if there is a constant $C\geq 1$ such that
$$ C^{-1}r^Q \leq \mathcal{H}^Q(\overline{B}(x,r)) \leq Cr^Q \text{ for all } r\leq \diam(X),$$
where $\mathcal{H}^Q$ denotes the $Q$-dimensional Hausdorff measure. 

The Ahlfors regular conformal dimension of $X$ measures the infimal dimension $Q$ of all Ahlfors regular quasisymmetric deformations of $X$:

\begin{definition}
The Ahlfors regular conformal dimension of a metric space $X$ is
\begin{equation}\label{eq:cdim}
 \cdim(X) = \inf\{ Q : Y \text{ is Ahlfors } Q\text{-regular and quasisymmetric to } X\}.
\end{equation}
\end{definition}
We refer the reader to \cite[Ch. 10]{He} for a precise definition of quasisymmetric mappings, and to \cite{BK04, MT} for more background on the Ahlfors regular conformal dimension, which we now discuss briefly. 

By definition, the Ahlfors regular conformal dimension and its variations are quasisymmetric invariants. They have thus played an important role in geometric group theory and quasiconformal geometry, and their properties are connected to many deep questions. We refer to \cite{MT} for a book-length account of many of these connections. 

In particular, it is a difficult problem to understand for which metric spaces the conformal dimension is actually achieved as a minimum. This problem is closely related to an approach to Cannon's conjecture initiated by Bonk and Kleiner \cite{BonkICM, BK04}.

As an example, it known that $\cdim(S)$ is strictly less than the Hausdorff dimension of the standard Sierpi\'nski carpet $S$ \cite{KL}, but not whether the infimum in \eqref{eq:cdim} is achieved by some Ahlfors regular space $Y$ when $X=S$. See \cite{MT, BK04} for additional details. (The exact value of $\cdim(S)$ is also a well-known open problem.)

Complicating the problem further, even if the conformal dimension of a subset $X\subseteq \RR^N$ is achieved by a space $Y$ quasisymmetric to $X$, there is no reason that $Y$ should be a subset of Euclidean space. Below, we show that this non-embedding phenomenon should be expected quite generally.

We need the following important result of Keith and Laakso. 
\begin{theorem}[Keith-Laakso \cite{KL}, Corollary 1.0.2]\label{thm:KL}
Let $Q \geq 1$ and let $X$ be a complete, Ahlfors $Q$-regular metric space. 

Then $\cdim(X)=Q$ if and only if there is a weak tangent of $X$ that contains a family of non-constant curves with positive $p$-modulus, for some $p\geq 1$.
\end{theorem}
(As remarked on \cite[p. 1279]{KL}, this version follows from the version stated there. The notion of a ``weak tangent'' is defined in section \ref{subsec:tangents}.)

As a consequence of Corollary \ref{cor:positivemodulus} and Theorem \ref{thm:KL}, we obtain:
\begin{corollary}\label{cor:cdim}
Let $Q\geq 1$ and let $X$ be a complete, Ahlfors $Q$-regular metric space, where $Q=\cdim(X)$. If $X$ admits a bi-Lipschitz embedding into some Euclidean space $\RR^N$, then there is a complete metric space $Z$ and a weak tangent of $X$ that is bi-Lipschitz equivalent to $Z\times \RR$.
\end{corollary}

Corollary \ref{cor:cdim} shows that we should not expect minimizers for conformal dimension to appear within Euclidean space except under quite special circumstances, i.e., in the presence of some form of splitting. We illustrate a more concrete special case here.

For this we introduce the terminology of linear connectedness: A metric space $X$ is \textit{linearly connected} (with constant $C\geq 1$) if every pair of points $x,y\in X$ can be joined by a compact, connected subset $E\subseteq X$ with $\diam(E) \leq C d(x,y)$. In particular, \textit{quasiconvex} spaces (in which every pair of points can be joined by a curve with length comparable to the distance between the points) are linearly connected.

\begin{corollary}\label{cor:carpet}
Let $X$ be a linearly connected metric space. Suppose that $1<\cdim(X)<2$ and that $\cdim(X)$ is achieved by a space $Y$. (In other words, $Y$ is quasisymmetric to $X$ and Ahlfors $Q$-regular with $Q=\cdim(X)=\cdim(Y)$.)

Then $Y$ admits no bi-Lipschitz embedding into any Euclidean space.
\end{corollary}

As an example, consider the classical Sierpi\'nski carpets $S_p$ for odd integers $p> 1$ (the most famous example being $S=S_3$). These are plane fractals formed by dividing the unit square into $p^{-1}\times p^{-1}$ squares and removing the middle square, then iterating this construction on the remaining squares. See \cite{BonkMerenkov} for details on this notation.

The spaces $S_p$ are all linearly connected and have $\cdim(S_p) \in (1,2)$ (see \cite[p. 595]{BonkMerenkov}). It is an open question whether the Ahlfors regular conformal dimensions of these spaces are actually achieved. Corollary \ref{cor:carpet} shows that they cannot be achieved by subsets of any Euclidean space.

We note that the linear connectedness condition cannot be removed from Corollary \ref{cor:carpet}: The product $\mathcal{C}\times [0,1]\subseteq \RR^2$ of the standard Cantor set with the unit interval is Ahlfors $Q$-regular (for $Q = 1+\frac{\log(2)}{\log(3)} \in (1,2)$) and known to be minimal for conformal dimension (by a theorem of Tyson \cite{Tyson}). However, it sits isometrically in $\RR^2$.

One may even make this example connected by taking its union with $[0,1]\times\{0\}$, showing that ``linearly connected'' cannot be replaced by ``connected'' in Corollary \ref{cor:carpet}. This example is also easily seen to have a weak tangent that is bi-Lipschitz equivalent to a product $Z\times \RR$.

\subsubsection{The slit carpet}
The slit carpet $\mathbb{M}$ is a metric space homeomorphic to the standard Sierpi\'nski carpet with a number of interesting properties. It was first proposed by Bonk and Kleiner and first studied in print by Merenkov \cite{Me}. Following \cite{Me}, we define the space as follows, mostly using notation from \cite{DEB}. We will be rather brief here, referring the reader to \cite{Me} or \cite{DEB} for more details.

Let $Q_0=[0,1]^2$ denote the unit square in $\RR^2$. For each dyadic subsquare $Q\subseteq Q_0$, let $s_Q$ denote a central vertical ``slit'' in $Q$ of half the side length. More specifically, if 
$$ Q =  [a2^{-k},(a+1)2^{-k}] \times [b2^{-k}, (b+1)2^{-k}],$$
then
$$ s_Q = \left[((2a+1)2^{-k-1},(4b+1)2^{-k-2}) ,((2a+1)2^{-n-1},(4b+3)2^{-k-2})\right] \subseteq Q.$$

Define now $M_0 = Q_0 = [0,1]^2$, and inductively set $$M_{k+1}=M_k \setminus \bigcup_{Q \text{ dyadic, }\side(Q)=2^{-(k+1)}} s_Q.$$ (See Figure \ref{fig:carpets}, borrowed from \cite{DEB}.)

We then define $\mathbb{M}_k$ as the completion of $M_k$ with respect to the shortest path metric $d_k$ on $M_k$, continuing to call this new complete metric on $\M_k$ by $d_k$. In other words, we ``cut along'' each slit in a square of scale $k$ or lower. Note that the $d_k$-diameter of each $\M_k$ is bounded by $3$.

\begin{figure}
\centering
\begin{subfigure}{0.3\textwidth}
  \centering
  \includegraphics[width=\linewidth]{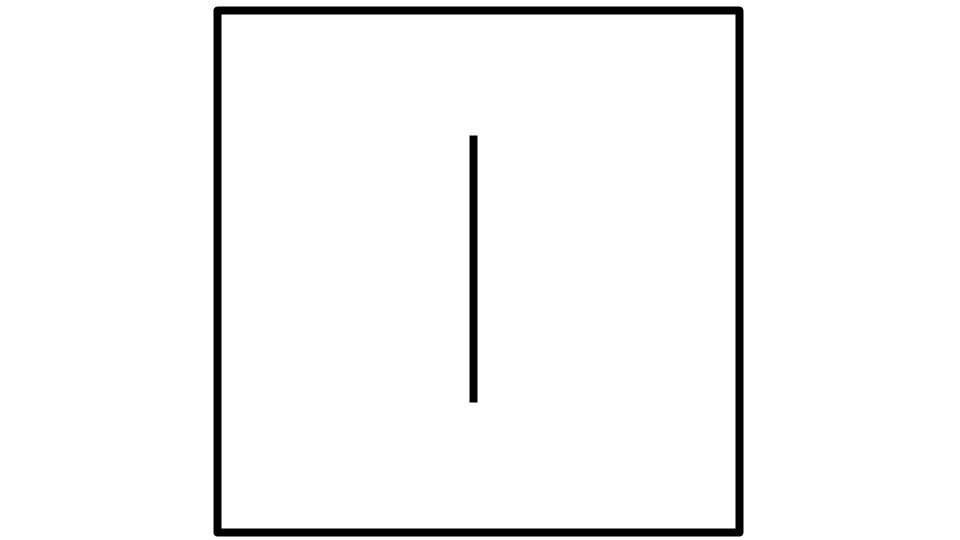}
  \caption{$M_1$}
  \label{fig:carpet1}
\end{subfigure}
\begin{subfigure}{0.3\textwidth}
  \centering
  \includegraphics[width=\linewidth]{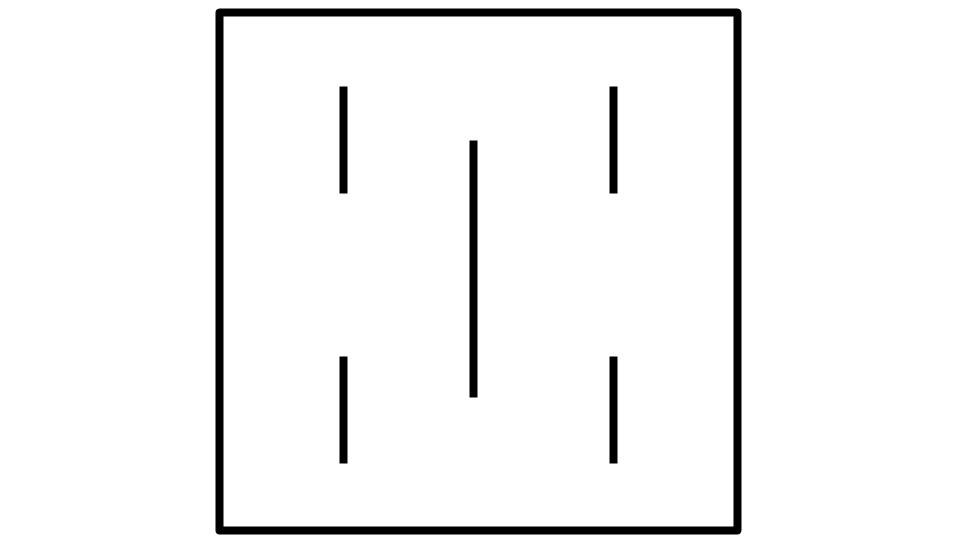}
  \caption{$M_2$}
  \label{fig:carpet2}
\end{subfigure}
\begin{subfigure}{0.3\textwidth}
  \centering
  \includegraphics[width=\linewidth]{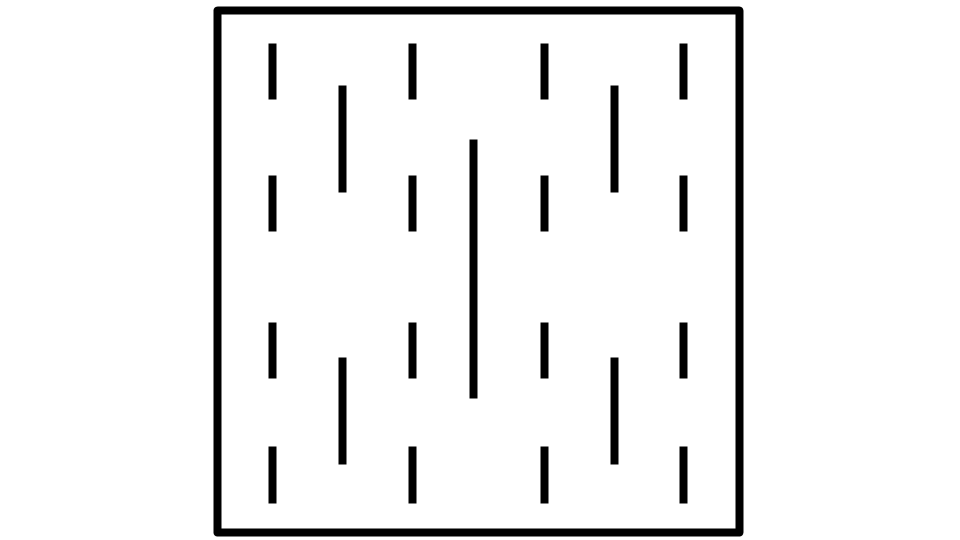}
  \caption{$M_3$}
  \label{fig:carpet3}
\end{subfigure}
\caption{The first three $M_i$.}
\label{fig:carpets}
\end{figure}

Merenkov observes in \cite{Me} that for each $k\leq j$, there is a $1$-Lipschitz mapping $\pi_{j,k}\colon \M_j \rightarrow \M_k$ obtained by identifying opposing points on slits of levels greater than $j$ corresponding to the same point in $\M_k$. These maps compose in the obvious way. We then define the Merenkov slit carpet $\M$ as the inverse limit of the system
$$ \M_0 \xleftarrow{\pi_{1,0}} \M_1 \xleftarrow{\pi_{2,1}} \M_2 \xleftarrow{\pi_{3,2}} \dots,  $$
equipped with the metric 
$$ d(x,y) = \lim_{k\rightarrow \infty} d_k(x_k, y_k)$$
for $x=(x_k)$ and $y=(y_k)$ such that $\pi_{k}(x_k)=x_{k-1}$, and similarly for $y_k$. Note that this is the limit of a bounded, increasing sequence.

In \cite{DEB}, the present authors answered a 1997 question of Heinonen and Semmes \cite[Question 8]{HS} by proving the following.

\begin{corollary}[Originally proven in Theorem 1.2 of \cite{DEB}]\label{cor:slitcarpet}
The slit carpet $\mathbb{M}$ does not admit a bi-Lipschitz embedding into any Euclidean space.
\end{corollary}
(Actually, the result in \cite{DEB} is more general, but the above corollary already answers \cite[Question 8]{HS}.)

We give an alternative proof of this fact below, based on Corollary \ref{cor:positivemodulus}. This proof is based on the idea that the vertical lines in Figure \ref{fig:carpets} form a positive modulus family of curves, while at the same time the slit carpet admits no infitesimal splitting. As in Corollaries \ref{cor:carpet} above and \ref{cor:Heisenberg} below, the main work is to verify the latter claim.

\subsubsection{The Heisenberg group}

The most well-known example of a doubling metric space that does not bi-Lipschitz embed into any Euclidean space is the first Heisenberg group $\bH$. This was first observed by Semmes \cite{Semmes} as a consequence of Pansu's version of Rademacher's theorem in Carnot groups \cite{Pansu}. There are also a number of other proofs that $\bH$ admits no bi-Lipschitz embedding into any Euclidean space \cite{MO}.

We give a new short proof of this non-embedding. Our proof is independent of Pansu's theorem (or the general metric measure space Rademacher theorem of Cheeger \cite{Ch}), and relies only on Theorem \ref{thm:mainthm} and some basic properties of the Heisenberg group. (Our proof does involve a ``blowup'' argument via Theorem \ref{thm:mainthm}, so it has that in common with the Pansu-Semmes approach.)

\begin{corollary}[Originally due to Semmes and Pansu]\label{cor:Heisenberg}
The Heisenberg group $\bH$ (with its Kor\'anyi or Carnot-Carath\'eodory metric) admits no bi-Lipschitz embedding into any Euclidean space. 
\end{corollary}

\subsection{Final introductory remarks and outline}

A few more remarks are in order concerning our main results.

First of all, Theorem \ref{thm:mainthm}, Corollary \ref{cor:bilip}, Corollary \ref{cor:positivemodulus}, and Corollary \ref{cor:cdim} are all completely false if one does not assume that the space $X$ lies inside, or admits a bi-Lipschitz embedding into, some Euclidean space. To be concrete, the Heisenberg group $\bH$ supports a doubling (even Ahlfors $4$-regular) measure supporting two independent Alberti representations, it supports curve families of positive modulus, and it is minimal for conformal dimension. (See \cite[Theorems 9.6, 9.27, 15.10]{He} and subsection \ref{subsec:heisenberg}.) However, no weak tangent of $\bH$ is bi-Lipschitz to some $Z\times \RR$; this is not difficult to prove given \cite[Theorem 7.2]{AK}. (One could actually prove Corollary \ref{cor:Heisenberg} along these lines, but we give a different argument that avoids the tools of \cite{AK,Pansu}.)

We note that Theorem \ref{thm:mainthm} is related to the main result of \cite{AlbMar}, which gives a ``partial differentiable structure'' for measures on Euclidean space that support independent Alberti representations. Here our focus is on the geometric tangent structure of subsets, rather than on differentiability properties of mappings, and our methods are different.

There are certainly further questions one could ask in our setting about the tangents of the measure $\mu$ in Theorem \ref{thm:mainthm}, rather than the tangents of the support $X$.  Related results for abstract metric measure spaces appear in \cite{CKS}. One could also ask about embeddings into infinite-dimensional Banach spaces. In the interest of keeping the present paper reasonably direct, we defer these questions to future work.

Using results of Schioppa \cite[Theorem 3.24 and Corollary 3.93]{Schioppa},  Theorem \ref{thm:mainthm} and Corollary \ref{cor:bilip} can be directly recast in terms of so-called Weaver derivations. In other words, if $X$, $\mu_0$, and $\mu$ are as in Theorem \ref{thm:mainthm}, but we assume that $\mu$ supports $k$ independent Weaver derivations rather than $k$ independent Alberti representations, the conclusion still holds. We refer the reader to \cite{Weaver} or \cite[Section 13]{HeNonsmooth} for more on Weaver derivations, and \cite{Schioppa} for more on the connection between Weaver derivations and Alberti representations.

\subsubsection{Outline of the paper}
In Section \ref{sec:prelim}, we give basic definitions and preliminary results. Within that section, subsection \ref{subsec:tangents} defines the different notions of tangents and our version of Preiss's principle, Proposition \ref{prop:basepoint}, is stated, though its proof is deferred to the appendix (Section \ref{sec:appendix}). Subsection \ref{subsec:alberti} defines Alberti representations and states the result of Bate, Proposition \ref{prop:bate}, that we will need, in addition to some other preliminary facts.

Theorem \ref{thm:mainthm} is then proven in Section \ref{sec:mainproof}. Section \ref{sec:modulus} contains Proposition \ref{prop:modmeasure} and Corollary \ref{cor:modmeasure}, which relate the notions of modulus and Alberti representations. All the corollaries listed in subsection \ref{subsec:corollaries} are  then proven in Section \ref{sec:corollaries}.

\section{Notation and preliminaries}\label{sec:prelim}

\subsection{Metric spaces and measures}

We write $(X,d)$ for a metric space, or just $X$ if the metric is understood. We use standard notation; in particular $B(x,r)$ and $\overline{B}(x,r)$ denote the open and closed balls, respectively, of radius $r$ centered at $x\in X$. If $c>0$ and $X=(X,d)$ is a metric space, then $cX$ denotes the metric space $(X, cd)$. We write $(X,\mu)$ for a metric measure space with a given Radon measure $\mu$ that is finite and positive on all balls, and where the metric $d$ is implied.

As usual, a map $f\colon X \rightarrow Y$ between two metric spaces is \textit{Lipschitz} if there is a constant $L\geq 0$ such that
$$ d(f(x),f(y)) \leq Ld(x,y) \text{ for all } x,y\in X.$$
In this case, we may also call $f$ \textit{$L$-Lipschitz} to emphasize the constant. The infimum of all $L$ such that $f$ is $L$-Lipschitz is denoted $\LIP(f)$. 

A map $f\colon X \rightarrow Y$ between two metric spaces is \textit{bi-Lipschitz} (or \textit{$L$-bi-Lipschitz}) if there is a constant $L\geq 1$ such that
$$ L^{-1}d(x,y) \leq d(f(x),f(y)) \leq Ld(x,y) \text{ for all } x,y\in X.$$
Two metric spaces are \textit{bi-Lipschitz equivalent} if there is a bi-Lipschitz surjection from one onto the other.

A \textit{pointed metric space} is a pair $(X,x)$, where $X$ is a metric space and $x\in X$ is a point (the ``basepoint''). Two pointed metric spaces are \textit{pointedly isometric} if there is an isometry between them that preserves basepoints.

Corollary \ref{cor:bilip} uses the notion of a product $X\times Y$ of two metric spaces $(X,d_X)$ and $(Y,d_Y)$. There are many equivalent ways to metrize this product. For concreteness, we take
$$ d_{X\times Y}((x,y),(x',y')) = (d_X(x,x')^2 + d_Y(y,y')^2)^{1/2}.$$
Any other natural choice would yield a bi-Lipschitz equivalent metric on $X\times Y$.

For a metric space $X$, we let $M(X)$ be the space of finite Borel measures on $X$. A non-trivial measure $\mu\in M(X)$ is \textit{doubling} if there is a constant $C\geq 0$ such that
$$ \mu(B(x,2r)) \leq C\mu(B(x,r)) \text{ for all } x\in X, r>0.$$
A metric space that supports a doubling measure must be a \textit{doubling metric space}: every ball of radius $r$ can be covered by $N$ balls of radius $r/2$, where $N$ is a fixed constant. (See \cite{He} for more on doubling metrics and measures.)

A complete doubling metric space is \textit{proper}: every closed ball is compact.

Lastly, we write $\HH^Q$ for the $Q$-dimensional Hausdorff measure on a metric space $X$ (with $X$ understood from context), and $\dim_H(X)$ for the Hausdorff dimension of $X$. See \cite[Section 8.3]{He} for definitions.

\subsection{Tangents of metric spaces and sets}\label{subsec:tangents}

The next few definitions will use the notion of \textit{pointed Gromov-Hausdorff convergence} of a sequence of pointed metric spaces. See \cite[Section 8]{BBI}. Versions are also given in \cite{DS, Keith, KL, GCD}, among other places.

\begin{definition}
Let $X$ be a metric space, and let $\lambda_i$ ($i\in\mathbb{N}$) be a sequence of positive real numbers. 
\begin{itemize}
\item If $x\in X$, $\lambda_i\rightarrow 0$, and the sequence of pointed metric spaces
$$ ( \lambda_i^{-1} X, x),$$
converges in the pointed Gromov-Hausdorff sense to a complete pointed metric space $(Y,y)$, then $(Y,y)$ is called a \textit{tangent of $X$ at $x$}. The collection of all tangents of $X$ at $x$ is written $\Tan(X,x)$.
\item If $\{x_i\}\subseteq X$, $0<\lambda_i < \diam(X)$ for each $i$, and the sequence of pointed metric spaces
$$ ( \lambda_i^{-1} X, x_i),$$
converges in the pointed Gromov-Hausdorff sense to a complete pointed metric space $(W,w)$, then $(W,w)$ is called a \textit{weak tangent of $X$}. The collection of all weak tangents of $X$ is written $\WTan(X)$.

\end{itemize}
\end{definition}
Technically speaking, elements of $\Tan(X,x)$ are not pointed metric spaces but rather pointed isometry classes, since the pointed Gromov-Hausdorff topology does not distinguish between isometric metric spaces; similarly, elements of $\WTan(X)$ are isometry classes. We tend to elide this distinction for notational convenience. Of course, if $(Y,y)\in\Tan(X,x)$ then  $Y\in\WTan(X)$.

If a space already sits inside an ambient Euclidean space, then one can more naturally take tangents by rescaling inside the Euclidean space and taking a limit in the pointed Hausdorff sense: If $\{A_j\}$ are subsets of $\RR^n$, then we say that $\{A_j\}$ converges to a closed set $A\subseteq \RR^n$ in the \textit{pointed Hausdorff sense} if 
$$ \lim_{j\rightarrow\infty} d_R(A_j, A) = 0 \text{ for all } R>0,$$
where
$$ d_R(A,B) = \max\left\{ \sup\{\dist(a,B): a\in A\cap B(0,R)\} , \sup\{ \dist(b,A): b\in B\cap B(0,R)\}\right\}.$$
Of course, one could take a sum rather than a max above (as done in \cite[Chapter 8]{DS}) and this would only change the definition by at most a factor of $2$.

\begin{definition}
Let $A$ be a subset of $\RR^n$ and $a\in A$. An \textit{intrinsic tangent of $A$ at $a$} is any closed set which is a pointed Hausdorff limit of sets of the form
$$ \lambda_j^{-1}(A - a), $$
where $\{\lambda_j\}$ is a sequence of positive real numbers tending to $0$ as $j\rightarrow\infty$.

The collection of all intrinsic tangents of $A$ at $a$ is written $\Tan_{\RR^n}(A,a)$.
\end{definition}

An important difference between tangents and intrinsic tangents is that $\Tan_{\RR^n}(A,a)$ may contain distinct elements that are isometric.

If one has a Lipschitz function defined on a subset of $\RR^n$, there is also a way to pass to a tangent (or ``blowup'') of the function simultaneously with the set. This is given in \cite[Chapter 8]{DS}.

\begin{definition}
Let $A$ be a subset of $\RR^n$ and $a\in A$. Let $f\colon A\rightarrow\RR^m$ be a Lipschitz function.

Suppose that $\{\lambda_j\}$ is a sequence of positive real numbers tending to $0$ as $j\rightarrow\infty$ and that
$$ \lambda_j^{-1}(A - a) \text{ converge to the closed set } \hat{A} \text{ in the pointed Hausdorff sense}. $$
Moreover, suppose that $\hat{f}:\hat{A}\rightarrow \RR^m$ is such that, whenever $x_j\in A$ and
$$ \lambda_j^{-1}(x_j-a) \rightarrow x\in \hat{A},$$
we then have
$$ \lambda_j^{-1}(f(x_j) - f(a)) \rightarrow \hat{f}(x).$$
We then call the pair $(\hat{A},\hat{f})$ an \textit{intrinsic tangent of $(A,f)$ at $a$} and write
$$ (\hat{A},\hat{f})\in\Tan_{\RR^n}(A,f,a).$$
\end{definition}

The following facts are all standard, well-known consequences of the compactness theorems for pointed Hausdorff and pointed Gromov-Hausdorff convergence. See, e.g., \cite[Lemmas 8.6 and 8.13]{DS}. We include this lemma simply as a summary of the basic facts that we will use.
\begin{lemma}\label{lem:tangentprops}
Let $X$ be a complete doubling metric space and $x\in X$. Let $A$ be a closed subset of $\RR^n$ and $a\in A$. Let $f:A\rightarrow B\subseteq \RR^m$ be Lipschitz. Let $\{\lambda_j\}$ be any sequence of positive real numbers tending to $0$.
\begin{enumerate}[(i)]
\item There is a subsequence $\lambda_{j_k}$ such that $( \lambda_{j_k}^{-1}X,x)$ converges in the pointed Gromov-Hausdorff sense to a doubling pointed metric space in $\Tan(X,x)$. In particular, $\Tan(X,x)\neq\emptyset$.
\item There is a subsequence $\lambda_{j_k}$ such that $\lambda_{j_k}^{-1}(A-a)$ converges in the pointed Hausdorff sense to an element $\hat{A}$ of $\Tan_{\RR^n}(A,a)$. In particular, $\Tan_{\RR^n}(A,a)\neq\emptyset$.
\item In the subsequence from (ii), we may also obtain that
\begin{itemize}
\item the functions $ \lambda_{j_k}^{-1}(f(\cdot) - f(a)) $ converge to a Lipschitz function $\hat{f}:\hat{A}\rightarrow\RR^m$ in the sense above, to yield $(\hat{A},\hat{f})\in\Tan_{\RR^n}(A,a,f)$,
\item the sequence $\lambda_{j_k}^{-1}(B-f(a))$ converges in the pointed Hausdorff sense to an element $\hat{B}$ of $\Tan_{\RR^m}(B,f(a))$, and
\item $\hat{f}(\hat{A})\subseteq\hat{B}$.
\end{itemize}
\item If $f$ is $L$-bi-Lipschitz, then so is $\hat{f}$. 
\item If $X$ and $A$ are Ahlfors $Q$-regular, then so are every element of $\WTan(X)$ and $\Tan_{\RR^n}(A,a)$.
\item If $\mu$ is a doubling measure on $A$ and $a$ is a point of density of a subset $A'\subseteq A$, then $\Tan_{\RR^n}(A',f,a)=\Tan(A,f,a)$.
\item If $Y\in \WTan(X)$, then $\WTan(Y) \subseteq \WTan(X)$.
\end{enumerate}
\end{lemma}
\begin{proof}
For (i), (ii), and (iii), see  \cite[Lemmas 8.6 and 8.13]{DS}. For (iv), see \cite[Lemma 8.20]{DS}. For (v), see \cite[Lemma 8.28]{DS}. For (vi) concerning the tangent spaces, see \cite[Lemma 9.6]{DS} or \cite[Proposition 3.1]{LD}; the extension to the tangent mappings is simple, as remarked in \cite{GCD}. For (vii), see \cite[Lemma 9.5]{DS}.
\end{proof}

We will need one more fact about tangents, a principle that appears in many different forms and goes back to Preiss \cite{Preiss}. Versions appear in, e.g., \cite{Mat05,AKL09,LD,GCD}. Informally, this is the principle that ``tangents with moved basepoints are still tangents''.

\begin{proposition}\label{prop:basepoint}
Let $A \subseteq \RR^n$ be a closed set supporting a doubling measure $\mu$. Let $f\colon A \rightarrow \RR^m$ be a Lipschitz mapping. Then for $\mu$-a.e. $a\in A$, the following holds:

For all $(\hat{A},\hat{f})\in\Tan_{\RR^n}(A,a,f)$ and all $b\in \hat{A}$, we have
$$ (\hat{A} - b, \hat{f}(\cdot+b)-\hat{f}(b)) \in \Tan_{\RR^{n}}(A,a,f).$$
\end{proposition}
The proof of Proposition \ref{prop:basepoint} is a minor modification of facts in the literature, and so postponed until the Appendix (Section \ref{sec:appendix}).

\subsection{Curves and fragments}\label{subsec:fragments}

The key objects in this paper are families of curves (or curve fragments) in metric spaces. We introduce some notation to discuss these objects. Our definitions and notation follow those in \cite{Schioppa} for the most part, with some minor changes.

Fix a separable, locally compact metric space $X$. A \textit{fragment} in $X$ is a bi-Lipschitz map $\gamma\colon C \rightarrow X$, where $C\subseteq \RR$ is compact and the one-dimensional Lebesgue measure $\mathcal{L}^1(C)$ is positive. We write $\Frag(X)$ for the collection of fragments in $X$. If $\gamma\in \Frag(X)$, then the domain $C$ of $\gamma$ is denoted $\dom(\gamma)$ and the image in $X$ is denoted $\im(\gamma)$. 

If $f: X \to \bR^m$ is any function, then we define $(f \circ \gamma)'(t) = \lim_{t' \to t, t' \in C} \frac{f(\gamma(t'))-f(\gamma(t))}{t'-t},$
when the limit exists and $t \in C=\dom(\gamma)$ is a density point. If $X \subset \bR^n$, then we simply write $\gamma'(t)$, when $f={\rm id}$ is the identity map. A \textit{density point} of a compact set $C \subset \bR$ is a $t \in C$, where $\lim_{h \to 0} \frac{\mathcal{L}^1(C \cap (t-h,t+h))}{2h} = 1$.

In Section \ref{sec:modulus}, we will also consider $\Curv(X)$, the collection of all  \textit{non-constant, Lipschitz} maps $\gamma\colon I \rightarrow X$,  where $I$ is a compact \textit{interval} in $\bR$ of positive length. Thus, elements of $\Curv(X)$ represent honest curves in $X$. We will also use the notation $\dom(\gamma)$ to denote the domain of an element $\gamma\in \Curv(X)$.

Note that neither $\Curv(X)$ nor $\Frag(X)$ is a subpsace of the other. We now discuss the appropriate topologies on $\Frag(X)$ and $\Curv(X)$, borrowing from \cite[Section 2]{Schioppa}. 

The spaces $\Frag(X)$ and $\Curv(X)$ both admit embeddings into the space $\Haus(\RR\times X)$ of non-empty compact subsets of $\RR\times X$, by
$$ \gamma \mapsto \{(t,\gamma(t)) : t\in \dom(\gamma)\}.$$
The space $\Haus(\RR\times X)$ is given the Hausdorff metric and the induced topology. If $X$ is complete, then so is $\Haus(\RR\times X)$.

Therefore, we topologize $\Frag(X)$ and $\Curv(X)$ as subspaces of $\Haus(\RR\times X)$. We note that these spaces are $\sigma$-compact if $X$ is proper.

\subsection{Line integrals and metric derivatives}
Let $X$ be a metric space and $\gamma\in\Curv(X)$. We denote by $\len(\gamma)$ the length of $\gamma$, as in \cite[Chapter 7]{He}. If $g\co X\rightarrow \RR$ is a Borel function, then $\int_\gamma g\,ds$ is defined as 
$$ \int_0^{\len(\gamma)} g(\tilde{\gamma}(t))\,dt,$$
where $\tilde{\gamma}$ is the arc length parametrization of $\gamma$; see \cite[Ch.7]{He}.

Following \cite[Definition 4.1.2]{AT}, the \textit{metric derivative} of $\gamma$ at a point $t\in \dom(\gamma)$ is
$$ d_\gamma(t) := \lim_{h\rightarrow 0} \frac{d(\gamma(t+h),\gamma(t))}{|h|},$$
whenever the limit exists. By \cite[Theorem 4.1.6]{AT}, $d_\gamma(t)$ does exist for a.e. $t\in\dom(\gamma)$, and
$$ \len(\gamma) = \int_{\dom(\gamma)} d_\gamma(t)\,dt.$$
It follows that the arc length parametrization $\tilde{\gamma}$ satisfies $d_{\tilde{\gamma}}(t)=1$ for a.e. $t\in\dom(\tilde{\gamma})$.

\subsection{Alberti representations}\label{subsec:alberti}

Fix a complete, locally compact, separable metric space $X$. Recall that $M(X)$ denotes the space of Radon measures on $X$. We equip $M(X)$ with the weak$^*$ topology arising from viewing $M(X)$ as the dual space of $C_c(X)$, the space of compactly supported continuous functions on $X$. See \cite[Assumption 2.3]{Schioppa} for details. Inside $M(X)$, we consider the subspace $P(X)$ consisting of probability measures. Note that elements of $P(X)$ are Borel. 

Fix a metric space $X$ and a measure $\mu\in M(X)$. The following definition is due to Bate \cite{Bate}, based on earlier work of Alberti \cite{Alberti}. In \cite{Schioppa}, the definition was clarified and modified slightly, and this is the definition we present below.

\begin{definition}\label{def:alberti} An \textit{Alberti representation} $\cA$ of $\mu$ is a pair $(P,\nu)$ where 
\begin{enumerate}[(i)]
\item $P$ is a Radon probability measure on $\Frag(X)$, 
\item $\nu \co \Frag(X) \to M(X)$ is a Borel map with $\nu_\gamma \ll \cH^1|_{\im(\gamma)}$ for each $\gamma\in \Frag(X)$, 
\item the measure $\mu$ can be represented as
$$\mu(A) = \int_{\Frag(X)} \nu_\gamma(A) dP(\gamma),$$
for each $A\subseteq X$ Borel,
\item and, for each Borel $A\subseteq X$ and compact interval $I\subseteq \bR$, the map $\gamma\mapsto \nu_\gamma(A \cap \gamma(\dom(\gamma) \cap I))$ is Borel. 
\end{enumerate}
\end{definition}

Note that, given the topologies defined above, the statement that the map  $\nu \co \Frag(X) \to M(X)$ is Borel means that
$$ \gamma \mapsto \int_X g(x) d\nu_\gamma(x)$$
is a Borel map from $\Frag(X)$ to $\RR$ for each $g\in C_c(X)$.

A \textit{cone} in $\bR^n$ is a set of the form 
$$ \Cone(w,t) := \{v\in \RR^n: v\neq 0 \text{ and } v\cdot w \geq t|v|\},$$
for some $w\in\mathbb{S}^{n-1}$ and $t\in \RR$.  Note that, for any $w\in\mathbb{S}^{n-1}$ and $t\leq -1$, $\Cone(w,t)=\RR^n\setminus\{0\}$.

\begin{remark}
Our definition of a cone departs slightly from those in \cite{Bate} and \cite{Schioppa}. In particular, our cones may have opening angle larger than $\pi$. It is clear that any cone of the types in \cite{Bate} and \cite{Schioppa} is a subset of a cone like one above.
\end{remark}

\begin{definition}\label{def:direction}
Fix a metric space $X$, a Lipschitz map $\phi:X\rightarrow \RR^n$, and a cone $C\subseteq \RR^n$.

A fragment $\gamma\in\Frag(X)$ is said to be \textit{in the $\phi$-direction of $C$} if $(\phi\circ\gamma)'(t)\in C$ for a.e. $t\in\dom(\gamma)$.

An Alberti representation $\cA = (P,\nu)$ of a measure $\mu\in M(X)$ is said to be \textit{in the $\phi$-direction of $C$} if $P$-a.e. $\gamma\in \Frag(X)$ is in the $\phi$-direction of $C$.
\end{definition}

\begin{definition}\label{def:independent}
Cones $C_1, \dots, C_k$ in $\RR^n$ are called \textit{independent} if each collection
$$ \{v_1, \dots, v_k : v_i\in C_i\} $$ 
is linearly independent.

A collection $\cA_1, \dots, \cA_k$ of Alberti representations of a measure $\mu\in M(X)$ is called \textit{$\phi$-independent}, for a Lipschitz $\phi:X\rightarrow \RR^m$, if there are independent cones $C_1, \dots, C_k$ in $\RR^m$ such that each $\cA_i$ is in the $\phi$-direction of $C_i$. 

We call a collection $\cA_1, \dots, \cA_k$ of Alberti representations \textit{independent} if they are $\phi$-independent for some Lipschitz map $\phi$ as above.
\end{definition}

A few remarks concerning this definition are in order.
\begin{remark}
If $X\subseteq \RR^n$ and $\mu\in M(X)$ supports $k$ $\phi$-independent Alberti representations, for some $\phi\colon X \rightarrow \RR^m$, then the map $\phi$ may be extended to a Lipschitz map $\phi\colon \RR^n\rightarrow \RR^m$ without altering the notion of $\phi$-independence.
\end{remark}

\begin{remark}
Traditionally (i.e., in \cite{Bate}), it is assumed that $m=k$ in Definition \ref{def:independent}, but we see no need to assume this, and in fact it will be occasionally convenient not to. 
\end{remark}

\begin{remark}
In the case $k=1$ of Definition \ref{def:independent}, one may take all of $\RR^m\setminus\{0\}$ as a single independent cone. Thus, a single Alberti representation $(P,\nu)$ is independent if and only if there is a Lipschitz map $\phi:X\rightarrow \RR^m$ such that $(\phi\circ \gamma)'(t)\neq 0$ for $P$-a.e. $\gamma\in\Frag(X)$ and a.e. $t\in\dom(\gamma)$.

In particular, if $X\subseteq \RR^n$, then a single non-trivial Alberti representation is automatically independent. Indeed, take $\phi$ to be the identity map. Since every $\gamma\in \Frag(X)$ is bi-Lipschitz, we have that
$$ (\phi \circ \gamma)'(t) = \gamma'(t) \neq 0$$
for all $\gamma\in\Frag(X)$ and a.e. $t\in \dom(\gamma)$.
\end{remark}

A last key fact for us will be the following result from \cite{Bate}. Essentially, one would like to know that a phenomenon which happens at almost every point along each curve in an Alberti representation actually happens almost everywhere in $X$. This is what the following result provides.  (See also the more general \cite[Proposition 2.9 ]{Bate}.) 

\begin{proposition}[Corollary 2.13 of \cite{Bate}]\label{prop:bate}
Let $X$ be a complete metric space with a Radon measure $\mu$. Let $\phi\colon X \rightarrow \RR^m$ be Lipschitz such that $\mu$ has $k$ $\phi$-independent Alberti representations. Let $f\colon X \rightarrow \RR^n$ be Lipschitz.

Then for $\mu$-a.e. $x\in X$, the following hold:
\begin{enumerate}[(i)]
\item There are $\gamma_1, \dots, \gamma_k\in \Frag(X)$ such that $\gamma_i(0)=x$ and $\gamma^{-1}(x)$ is a density point of $\dom(\gamma)$.
\item The derivatives $(\phi \circ \gamma_i)'(0)$ exist and form a linearly independent set in $\RR^m$.
\item The derivatives $(f\circ \gamma_i)'(0)$ exist.
\end{enumerate}
\end{proposition}
We briefly note that Bate assumes that $k=m$ and $n=1$ in the cited result, but the proof using \cite[Lemma 2.8 and Proposition 2.9]{Bate} works in this generality.

\subsection{Connecting to other measures defined on curve families}

In Section \ref{sec:modulus}, we will need to connect Alberti representations to a related type of measure defined on $\Curv(X)$.

\begin{proposition}\label{prop:decomposition}
Let $X$ be a proper metric space, $\phi: X \to \bR^n$ bi-Lipschitz, and $C \subset \bR^n$ a cone. Suppose that $P$ is a Radon measure\footnote{While the notation may suggest so, this measure need not be a probability measure.} on $\Curv(X)$, so that for $\nu$-almost every $\gamma$, $(\phi \circ \gamma)'(t) \in C$ or $d_\gamma(t) = 0$ for almost every $t \in \dom(\gamma)$.

If the Borel measure defined by
$$\mu(A)= \int_{\Curv(X)} \int_\gamma 1_A ~ds ~dP$$
is locally finite (hence Radon), then it admits an Alberti representation in the $\phi$-direction of $C$.
\end{proposition}

\begin{remark} The literature is rife with different versions of Alberti representations, see  \cite{AlbMar,Bate,Schioppa,CKS} for some of them. Modifications of this argument can be used to show that, roughly speaking, if one has a representation in one of these senses, then one has also a representation in any other sense.

We briefly remark that the cones considered in \cite{Bate} are slightly different from ours, but the proof applies for both notions of cone.
\end{remark}

\begin{proof}

By \cite[Corollary 5.8]{Bate} (see the ``in particular...'' statement), we can decompose $X=A \cup N$, where $\mu|_{A}$ admits an Alberti representation in the $\phi$-direction of $C$, and 
$$ \HH^1(\im(\gamma) \cap N)=0$$
for every $\gamma\in\Frag(X)$ in the $\phi$-direction of $C$.

If we can show that $\mu(N)=0$, then $\mu$ restricted to the full-measure set $A$ supports an Alberti representation in the $\phi$-direction of $C$, and this completes the proof. To establish this, we will show that for $P$-almost every curve $\gamma$ we have
$$\int_\gamma 1_N ~ds =0.$$

First, for $P$-almost every curve $\gamma: I \to X$, and almost every $t \in I$ we have $(\phi \circ \gamma)'(t) \in C$ or $d_\gamma(t) =0$. Let $\gamma$ be any curve with such properties, and let $\tilde{\gamma}:\tilde{I} \to X$ be its arc length reparametrization. Then $(\phi \circ \tilde{\gamma})'(t) \in C$ for almost every $t \in I$. 

By \cite[Lemma 4]{Kirchheim}, we can find compact sets $K_j$ such that $\tilde{I} = \bigcup_j K_j \cup S$, $|S|=0$, and
$$\tilde{\gamma}_j := \tilde{\gamma}|_{K_j} \text{ is bi-Lipschitz.}$$ 
It follows that $\tilde{\gamma}_j\in\Frag(X)$ and in the $\phi$-direction of $C$, and hence $\HH^1(\im(\tilde{\gamma}_j) \cap N)=0$ for each $j$.

Thus,
\[
\int_\gamma 1_N ~ds = \int_{\tilde{I}} 1_N(\tilde{\gamma}(t))~dt = \sum_{j} \int_{K_{j}}  1_N(\tilde{\gamma}(t))~dt = 0,
\]
which completes the proof.

\end{proof}

\section{Proof of Theorem \ref{thm:mainthm}}\label{sec:mainproof}

In this section, we prove Theorem \ref{thm:mainthm}. The proof requires a few lemmas.

\begin{lemma}\label{lem:ARblowup}
Let a closed set $X\subseteq \RR^n$ support a doubling measure $\mu_0$. Let $\mu\ll\mu_0$ support $k$ $\psi$-independent Alberti representations, for some Lipschitz $\psi\colon \RR^n \rightarrow \RR^m$.

Then for $\mu$-a.e. $x\in X$, there are linearly independent vectors $v_1, \dots, v_k$ such that the following holds: For every $Y\in\Tan_{\RR^n}(X,x)$, every $y\in Y$, and every $i\in\{1, \dots, k\}$, there is a line through $y$ in direction $v_i$ that is contained in $Y$.
\end{lemma}
\begin{proof}
We apply Proposition \ref{prop:bate}, in the case $\phi=\psi$ and $f$ is the inclusion $X\rightarrow \RR^n$.

This tells us that, at $\mu$-a.e. $x\in X$, there are $\gamma_1, \dots, \gamma_k\in \Frag(X)$ such that the following hold:
\begin{enumerate}[(i)]
\item For each $1\leq i \leq k$, we have $\gamma_i(0)=x$ with $0$ a density point of $\dom(\gamma)$. 
\item The vectors $(\psi \circ \gamma_1)'(0), \dots, (\psi \circ \gamma_k)'(0)$ are linearly independent in $\RR^m$.
\item For each $1\leq i \leq k$, $\gamma_i'(0)$ exists.
\end{enumerate}

Fix an $x\in X$ where the above hold and where the conclusion of Proposition \ref{prop:basepoint} holds. Let
$$ w_i =  (\psi \circ \gamma_i)'(0)\in \RR^k.$$

Let $v_i=\gamma_i'(0)\in \RR^n$. Note that $v_i\neq 0$ as $\gamma_i$ is bi-Lipschitz, and $w_i\neq 0$ by (ii).

Consider an arbitrary tangent
$$ (Y,\hat{\psi})\in\Tan_{\RR^n}(X,\psi,x),$$
subject to the sequence of scales $\lambda_k \rightarrow 0$.

By Lemma \ref{lem:tangentprops} (items (ii), (iii), and (vi)), we may pass to a subsequence of $\{\lambda_j\}$ subject to which the following tangents also exist (for each $1\leq i \leq k$):

$$ (\RR, L_i) \in \Tan_{\RR}(\dom(\gamma_i),\gamma_i,0) \text{ with } L_i(\RR)\subseteq Y, \text{ and } $$
$$ (\RR,\hat{\psi}\circ L_i) \in \Tan_{\RR}(\dom(\gamma_i),\psi\circ\gamma_i,0).$$
Moreover, since $\gamma_i$ and $\psi\circ \gamma_i$ are differentiable at $0$, their tangent maps $L_i$ and $\hat{\psi}\circ L_i$ are linear. In particular, recalling $\gamma_i'(0)=v_i$ and $(\psi\circ \gamma_i)'(0)=w_i$, we have the following properties of $L_i$:
\begin{equation}\label{eq:Li}
 L_i(t) = tv_i\in Y \text{ and } \psi(L_i(t)) = tw_i \text{ for all } t\in\RR.
\end{equation}
By definition, we also have $0\in Y$ and $\hat{\psi}(0)=0$.

To summarize, the above argument shows that for \textit{every} element $(Y,\hat{\psi})\in\Tan(X,x,\psi)$, there is a line $L_i$ through $0$ with the properties in \eqref{eq:Li}.

Consider again an arbitrary  $(Y,\hat{\psi})\in\Tan(X,x,\psi)$. Proposition \ref{prop:basepoint} therefore says that for every $y\in Y$, the pair $(Y-y, \hat{\psi}(\cdot+y)-\hat{\psi}(y))$ is also an element of $\Tan_{\RR^n}(X,\psi,x)$. 

This implies that for every $y\in Y$ and $i\in\{1, \dots, k\}$, there is a function 
$$L_i^y\colon \RR\rightarrow Y$$
such that
$$ L^y_i(t) = y+tv_i \text{ and } \psi(L_i(t)) = \hat{\psi}(y)+ tw_i \text{ for all } t\in\RR.$$
In other words, $L^y_i$ is the parametrization of a line through $y$ in direction $v_i$ (contained in $Y$), whose composition with $\hat{\psi}$ parametrizes a line through $\hat{\psi}(y)$ in direction $w_i$.

It remains to show that the vectors $v_i$ are linearly independent. Suppose to the contrary that there was a non-trivial linear combination 
$$ \sum_{i=1}^k a_i v_i = 0.$$
Let $y_0=0\in Y$ and $p_0=0$. For $i=1, \dots, k+1$, inductively set
$$ y_i = L^{y_{i-1}}_i(a_i) = y_{i-1} + a_{i}v_{i}$$
and
$$ p_i = \hat{\psi}(y_i).$$
Note that $p_i = p_{i-1} + a_{i}w_{i}$ for each $i=1,\dots,k$ by the properties of $L^y_i$ above.

Then $y_{k}=0$. This implies that $p_{k} = \hat{\psi}(y_k)=0$. On the other hand
$$ p_{k} = \sum_{i=1}^k a_i w_i.$$
This contradicts the linear independence of the vectors $w_i$.

\end{proof}

\begin{lemma}\label{lem:factor}
Let $Y\subseteq\RR^n$ be a closed set. Let $v_1, \dots, v_k$ be linearly independent in $\RR^n$. Suppose that, for each $y\in Y$, there are $k$ lines
$$ L_i = \{y + tv_i : t\in \RR\}$$
that pass through $y$ and are contained in $Y$.

Then $Y = Z \times V$, where $V = \text{span}(\{v_1, \dots, v_k\})$ and $Z\subseteq V^{\bot}$ is a closed set.
\end{lemma}
\begin{proof}
Let $V = \text{span}(\{v_1, \dots, v_k\})$, and let 
$$ Z = \text{proj}_{V^\bot}(Y),$$
the projection of $Y$ to the orthogonal complement of $V$.

We now claim that $Y=Z\times V$. Certainly $Y\subseteq Z\times V$, by definition of orthogonal projection. For the other direction, suppose $p\in Z\times V$. Then 
$$p = \text{proj}_{V^\bot}(y) + a_1 v_1 + \dots + a_k v_k,$$
where $y\in Y$ and $a_i\in\RR$. Let $z=\text{proj}_{V^\bot}(y)\in Z$. 

The point $y_0=y$ is in $Y$. For $i=1, \dots, k$, we inductively set $y_i = y_{i-1} + a_i v_i$. By assumption, each point $y_i$ is in $Y$. Note that the last point $y_k$ is equal to $p$. Hence $p\in Y$, which proves that $Z\times V \subseteq Y$.

Lastly, we argue that $Z$ is closed. Indeed, if $z_n$ is a sequence in $Z$ converging to $z\in\RR^n$, then the points
$$(z_n,0)\in Y\subseteq V^{\bot}\times V = \RR^n$$
converge to $(z,0)\in Y$, since $Y$ is closed. It follows that $z\in Z$. 
\end{proof}

\begin{proof}[Proof of Theorem \ref{thm:mainthm}]

Let $X\subseteq \RR^n$ be a closed set admitting a doubling Radon measure $\mu_0$. Let $\mu$ be a measure absolutely continuous to $\mu_0$ that admits $k$ $\phi$-independent Alberti representations.

Let $x$ be a point at which the conclusion of Lemma \ref{lem:ARblowup} holds (a set of points that has full $\mu$-measure). Let $v_1, \dots, v_k$ be the associated linearly independent vectors in $\RR^n$.

Let $Y\in \Tan_{\RR^n}(X,x)$. Then each point $y\in Y$ admits $k$ lines $L_1, \dots, L_k$ through $y$, in directions $v_i$, that are contained in $Y$.

By Lemma \ref{lem:factor}, this implies that $Y= Z\times V$ for a $k$-dimensional subspace $V = \text{span}(\{v_1, \dots, v_k\})$ and some closed set $Z\subseteq V^{\bot}\subseteq \RR^n$. This completes the proof.
\end{proof}

\section{Modulus and Alberti representations}\label{sec:modulus}

In this section, we relate Alberti representations to the more classical notion of the modulus of a family of curves. The main result in this section is Proposition \ref{prop:modmeasure}, which may be of independent interest. (See Remark \ref{rmk:modmeasure} for more on the provenance of this result.)

We first recall the definition of the modulus of a family of curves. It is worth noting, that for us $\Curv(X)$ consists only of curves with Lipschitz parametrizations, while traditionally Modulus is defined for an \textit{a priori} larger class of collections of $\gamma\colon I \to X$, which are merely continuous. However generality is not lost, as the modulus of non-rectifiable curves vanishes by convention, and rectifiable curves can be reparametrized as Lipschitz curves without affecting the modulus. 

\begin{definition}\label{def:modulus}
Let $X$ be a metric space with a Radon measure $\mu$, let $\Gamma\subseteq \Curv(X)$, and let $p\geq 1$.

A Borel measurable function $\rho \co X \to [0,\infty]$ is called \textit{admissible} for $\Gamma$ if  $\int_\gamma \rho ~ds \geq 1$ for each $\gamma \in \Gamma$. We set $\cA(\Gamma)$ to be the collection of all admissible functions for $\Gamma$.

The \textit{$p$-modulus} of $\Gamma$, with respect to the measure $\mu$, is denoted
\begin{equation}\label{eq:moddef}
\Mod_p(\Gamma,\mu) = \inf_{\rho \in \cA(\Gamma)}\int_X \rho^p ~d\mu.
\end{equation}
\end{definition}

For our duality argument, we will need to work with continuous functions $\rho$. Thus, we define $\Mod_p^c(\Gamma,\mu)$ by replacing the infimum in \eqref{eq:moddef} with the infimum over all admissible $\rho\colon X \to [0,\infty)$ that are in addition continuous with compact support. In general, $\Mod_p^c(\Gamma,\mu)$ may be larger that $\Mod_p(\Gamma,\mu)$, and our first goal is to identify an assumption under which they are equal. 
 
We will need the following basic continuity fact both for the equality of $\Mod_p$ and $\Mod_p^c$ and for the duality argument below. It is a version of \cite[Proposition 4]{Keith} in our topology.

\begin{lemma}\label{lem:lowersemicont} Let $\rho_n: X \to \bR$ be an increasing sequence of lower semi-continuous functions convering to $\rho: X \to \bR$. If $\gamma_n:[a,b] \to X$ converge uniformly to $\gamma:[a,b] \to X$, or if $\gamma_n \in \Curv(X)$ converge to $\gamma \in \Curv(X)$, then
$$\liminf_{n \to \infty} \int_{\gamma_n} \rho_n ~ds \geq \int_\gamma \rho ~ds.$$ Further, the map $\gamma \mapsto \int_\gamma \rho ~ds$ is lower semi-continuous on $\Curv(X)$. 
\end{lemma}
\begin{proof} First, we reduce the case of $\gamma_n \in \Curv(X)$ to the case of uniform convergence with a common domain. Suppose that $\gamma_n$ converge to $\gamma$ in $\Curv(X)$. Then their graphs $\Gamma(\gamma_n)$ converge to $\Gamma(\gamma)$ in the Hausdorff metric on subsets $X \times \bR$. This claim still holds if we reparametrize each curve by an increasing affine map to have domain $[0,1]$.  After such reparametrization, by \cite[Theorem 1]{WCW}, we see that these reparametrizations converge uniformly.

Without loss of generality, we thus assume that $a=0$ and $b=1$, and that the parametrized curves converge uniformly. Note that for each fixed lower semicontinuous $g\colon X \rightarrow \RR$, the map $\gamma \mapsto \int_\gamma g \,ds$ is lower semicontinuous on the space of curves $\gamma\colon [0,1]\rightarrow X$ with the topology of uniform convergence. See, for example, \cite[Lemma 2.2]{JJRRS}.

Now, fix $N \in \bN$. Then,  for $n \geq N$ we get that
$$\int_{\gamma_n} \rho_n ~ds \geq \int_{\gamma_n} \rho_N ~ds.$$
Taking a limit inferior on both sides and using the lower semi-continuity noted above, we get
$$\liminf_{n \to \infty} \int_{\gamma_n} \rho_n ~ds \geq \int_{\gamma} \rho_N ~ds.$$
Finally, sending $N \to \infty$ and using dominated convergence completes the claim.

The latter claim on lower semi-continuity is a restatement of the first claim, by setting $\rho_n=\rho$ and assuming that $\gamma_n \in \Curv(X)$ converge to $\gamma \in \Curv(X)$.
\end{proof}

We will need the following Lemma. The proof is almost identical to the ones in  \cite{HK, Keith,exnerova2019plans}, however, for completeness and since there is significant variation in the literature on terminology, we recall the main steps of the argument. We will follow the scheme of the proof of \cite[Proposition 6]{Keith}, highlighting the main differences along the way, and the reader may consult it for additional details.

\begin{lemma}\label{lem:contmod} If $\Gamma \subset \Curv(X)$ is compact, then $\Mod_p(\Gamma,\mu) = \Mod_p^c(\Gamma,\mu)$
\end{lemma}
\begin{proof} Since continuous admissible functions are also Borel admissible, then $\Mod_p(\Gamma,\mu) \leq \Mod_p^c(\Gamma,\mu)$. We proceed to show the reverse inequality by an approximation argument. If $\Mod_p(\Gamma,\mu)= \infty$, there is nothing to prove. Otherwise, take any $\rho$ admissible for $\Mod_p(\Gamma,\mu)$ with finite $L^p$-norm. Fix $\epsilon>0$. By the Vitali-Caratheodory theorem, we can approximate any Borel $\rho$ from above by a lower semi-continuous function $\widetilde{\rho}$ with $\|\widetilde{\rho}\|_p^p \leq \|\rho\|_p^p+\epsilon$. 

Since $\Gamma$ is compact, we must have a some bounded ball $B(x,R) \subset X$ which contains all of the curves. In contrast to \cite[Proposition 6]{Keith}, we do not need to adjust $\widetilde{\rho}$ further.

Next, we approximate from below. Let $\widetilde{\rho_n} \nearrow \widetilde{\rho}$ be a sequence of continuous functions with compact support forming an increasing sequence and converging to $\rho$.

As in \cite[Proposition 6]{Keith}, it suffices to prove that
\begin{equation}\label{eq:limsup}
1 \leq \limsup_{n \to \infty} \inf_{\gamma \in \Gamma} \int_{\gamma} \widetilde{\rho_n}~ds.
\end{equation}
Indeed, in this case the function $(1-\epsilon)^{-1}\widetilde{\rho}_n$ would be admissible for $\Gamma$ for sufficiently large $n$, forcing
$$ \Mod_p^c(\Gamma,\mu) \leq (1-\epsilon)^{-p}(\|\rho\|_p^p+\epsilon).$$
Since this holds for all Borel admissible $\rho$ and $\epsilon>0$, the desired inequality immediately follows.

Choose curves $\gamma_n \in \Gamma$ so that
$$\liminf_{n \to \infty} \int_{\gamma_n} \widetilde{\rho_n} ~ds \leq \limsup_{n \to \infty} \inf_{\gamma \in \Gamma} \int_{\gamma} \widetilde{\rho_n}~ds.$$

Since $\Gamma$ is compact, there exists a subsequence of $\gamma_n$ which converges in $\Curv(X)$ to some $\gamma$, which we continue to label $\gamma_n$. (Unlike \cite[Proposition 6]{Keith}, we do not need to reparametrize, or estimate lengths, as we are assuming compactness.)  By using Lemma \ref{lem:lowersemicont}, the fact that $\rho \leq \widetilde{\rho}$, and admissibility, we obtain \eqref{eq:limsup}:

$$1 \leq \int_\gamma \rho ~ds \leq \int_\gamma \widetilde{\rho} ~ds \leq \liminf_{n \to \infty} \int_{\gamma_n} \widetilde{\rho_n} ~ds \leq \limsup_{n \to \infty} \inf_{\gamma \in \Gamma} \int_{\gamma} \widetilde{\rho_n}~ds.$$

\end{proof}

A few standard, well-known facts about modulus will be used repeatedly: (See \cite[Ch. 7]{He}.)
\begin{equation}\label{eq:monotone}
\text{If } \Gamma'\subseteq\Gamma \text{ then } \Mod_p(\Gamma',\mu)\leq\Mod_p(\Gamma,\mu).
\end{equation}
\begin{equation}\label{eq:subadditive}
\Mod_p\left( \cup_{i=1}^\infty \Gamma_i,\mu\right)\leq\sum_{i=1}^\infty \Mod_p(\Gamma_i,\mu).
\end{equation}

We will also need the following notion, that translates a measure on a curve family into a measure on a space.
\begin{definition}\label{def:curvemeasure}
Let $P$ be a Radon probability measure on $\Curv(X)$. It then defines an induced Radon measure $\eta_P$ on $X$ by

$$\eta_P(E) = \int_{\Curv(X)} \int_\gamma 1_E ~ds ~dP(\gamma).$$

\end{definition}

Our main result in this section is the following. 
\begin{proposition}\label{prop:modmeasure}
Let $X$ be a proper metric space with a Radon measure $\mu$. Let $p\in[1,\infty)$, and let $q \in (1,\infty]$ be the dual exponent, so that $\frac{1}{p}+\frac{1}{q}=1$. Let  $\Gamma\subseteq \Curv(X)$ be compact.

If $\Mod_p(\Gamma,\mu) \in (0,\infty)$, then there is a Radon probability measure $P$ on $\Curv(X)$, so that
$$\eta_P = f\mu,$$
with $f \in L^q(\mu)$ and
\begin{equation}\label{eq:density}
||f||_{L^q} = \Mod_p(\Gamma,\mu)^\frac{-1}{p}.
\end{equation}
\end{proposition}

\begin{remark}\label{rmk:modmeasure}
The proof of Proposition \ref{prop:modmeasure} that we give below is essentially already contained in work of the second-named author and his collaborators in \cite[Section 3]{Durand}, though in a less general context.

This result can also be deduced from \cite{Amb} if $p>1$, where it was shown that modulus is dual to probability measures in a much more general context (for $p>1$).

For completeness, and since the proof is short, we include an argument for all $p\geq 1$ in the case where $\Gamma$ is a compact family of curves.
\end{remark}

Recall that $C_c(X)$ denotes the space of continuous functions with compact support, with the uniform topology.

The main technical method in the proof of Proposition \ref{prop:modmeasure} is to express the Lagrangian in a given form and apply the following minimax principle on it. See also \cite[Section 9]{Rudin} for another version.

\begin{theorem}[\em Sion's minimax theorem, Corollary 3.3 in 
\cite{sion}\footnote{The formulation here is slightly simplified.}]
\label{thm:minmax} 
Suppose that
\begin{enumerate}[(i)]
\item $G$ is a convex subset of some topological vector space,
\item $K$ is a compact convex subset of some topological vector space, and
\item $F \co G \times K \to \bR$ satisfies
\begin{enumerate}[(a)]
\item $F(\cdot, y)$ is convex and lower semi-continuous on $G$ for every $y \in K$,
\item $F(x,\cdot)$ is concave and upper semi-continuous on $K$ for every $x \in G$.
\end{enumerate}
\end{enumerate}
Then we have the equality
\[
\sup_{y \in K} \ \inf_{x \in G} \ F(x,y) \, = \inf_{x \in G} \ \sup_{y \in K}\  F(x,y).
\]
\end{theorem}

\begin{proof}[Proof of Proposition \ref{prop:modmeasure}]
Note that, since $\Gamma$ is compact, all curves of $\Gamma$ must lie in a bounded subset of $X$, which must be compact and of finite $\mu$-measure. Thus, it suffices to assume that $X$ is compact and $\mu(X)<\infty$ in the proof.

Let $K$ be the set of Radon probability measures on $\Gamma$. Then $K$ is a compact, convex subset of the space of Radon measures on $\Curv(X)$, equipped with the topology of weak* convergence. Let $G=\{g \co X \to [0,1]\} \cap C_c(X)$.

Consider the functional $\Phi \co G \times K \to \bR $
$$\Phi(g,P) = \|g\|_{L^p(\mu)} - \Mod_p(\Gamma,\mu)^\frac{1}{p}\int_X g ~d\eta_P.$$
Lemma \ref{lem:lowersemicont} gives the upper semi-continuity of $\Phi(g, \cdot)$. The other conditions for $\Phi=F$ in Theorem \ref{thm:minmax} are verified as in \cite[Theorem 3.7]{Durand} . 

Let $g \in G$ be any fixed function. Since $g$ has compact support, it has finite $L_p(\mu)$-norm. Fix $\epsilon>0$. The function
$$\overline{g}_\epsilon = \frac{\Mod_p(\Gamma,\mu)^\frac{1}{p} g}{(1+\epsilon)||g||_{L^p}}$$
cannot be admissible, and thus there is a curve $\gamma_\epsilon \in \Gamma$ with
$$\int_{\gamma_\epsilon} \overline{g}_\epsilon ~ds \leq 1.$$
Then,
$$\int_{\gamma_\epsilon} g ~ds \leq (1+\epsilon)||g||_{L^p}\Mod_p(\Gamma,\mu)^\frac{-1}{p}.$$
Let $P_{\gamma_\epsilon} = \delta_{\gamma_\epsilon}$, the Dirac measure supported on $\gamma_\epsilon \in \Curv(X)$. Then
$$\Phi(g,P_{\gamma_\epsilon}) \geq -\epsilon ||g||_{L^p(\mu)}.$$
Therefore, sending $\epsilon \to 0$, we obtain
$$\inf_{g \in G} \sup_{P \in K} \Phi(g,P) \geq 0.$$

By Theorem \ref{thm:minmax}, we get 
$$\sup_{P \in K} \inf_{g \in G} \Phi(g,P) \geq 0.$$

Thus, there is a sequence of measures $P_\epsilon\in K$ so that
$$\Phi(g,P_\epsilon) \geq -\epsilon,$$
for each $g \in G$. By upper semi-continuity in $P_\epsilon$, and weak compactness, we can extract a weak limit $P$ for which it holds that
$$\Phi(g,P) \geq 0,$$
for each $g \in G$. In particular
\begin{equation}\label{eq:riesz}
\int_X g ~d\eta_P \leq \Mod_p(\Gamma,\mu)^\frac{-1}{p}||g||_{L^p(\mu)}.
\end{equation}

Thus, the functional
$$g \to \int_X g ~d\eta_P$$ extends to a $L_p$ bounded linear functional. Then, by the Riesz representation theorem we have
$$\eta_P = f \mu,$$
with the bound
$$||f||_{L^q} \leq \Mod_p(\Gamma,\mu)^\frac{-1}{p}$$
following from Estimate \eqref{eq:riesz}.

On the other hand, by Lemma \ref{lem:contmod} we can find a sequence of continuous admissible $\rho_i$ so that 
$$\lim_{i \to \infty} \int \rho_i^p ~d\mu = \Mod_p(\Gamma,\mu).$$

By a standard approximation argument of continuous functions by simple functions,  and the fact that $\rho_i$ is admissible, we obtain that
$$\int_X \rho_i ~d\eta_P = \int_{\Curv(X)} \int_\gamma \rho_i ~ds ~dP \geq 1.$$
Therefore, we get
$$1 \leq \int_X \rho_i ~d\eta_P  \leq \Mod_p(\Gamma,\mu)^\frac{-1}{p}||\rho_i||_{L^p}.$$
Sending $i \to \infty$, we get
\begin{equation}\label{eq:inlimit}
1 \leq \lim_{i \to \infty} \int_X \rho_i ~d\eta_P \leq  \Mod_p(\Gamma,\mu)^\frac{-1}{p}  \Mod_p(\Gamma,\mu)^\frac{1}{p} =1.
\end{equation}
Thus, since we get equality in the limit, $\Mod_p(\Gamma,\mu)^\frac{-1}{p} $ equals the norm of the functional $g \to \int g ~d\eta_P$, and thus $||f||_{L^q} = \Mod_p(\Gamma,\mu)^\frac{-1}{p}$.
\end{proof}

\begin{remark} Further details could be obtained. If $p>1$, then the sequence in the last paragraph $L^p$-converges  $\rho_i \to \rho^*$, where $\rho^*$ is an admissible function for $\Gamma \setminus \Gamma'$ where $\Gamma'$ has modulus zero, see e.g. \cite{Ziemer}. Then, $\rho^*$ plugged into Estimate \ref{eq:inlimit} (without the limit), yields 
\begin{equation}\label{eq:equality}
(\rho^*)^p = \Mod_p(\Gamma,\mu)^\frac{p+q}{p}f^q
\end{equation}
almost everywhere. Further, $\int_\gamma \rho^* ~ds=1$ for $P$-almost every $\gamma$, where $P$ is the probability measure coming from the statement.

\end{remark}

As an immediate corollary of Propositions \ref{prop:modmeasure} and \ref{prop:decomposition}, we explicitly point out the connection between modulus and Alberti representations:
\begin{corollary}\label{cor:modmeasure}
Let $X$ be a proper metric space with a Radon measure $\mu$. Suppose that $x \in X$ admits a family $\Gamma\subseteq \Curv(X)$ with positive modulus, that is
$$ \Mod_p(\Gamma,\mu)>0 \text{ for some } p\geq 1.$$
Then there is a non-zero measure on $X$, absolutely continuous to $\mu$, which admits an Alberti representation.  Furthermore, if $\phi : X \to \bR^n$ is bi-Lipschitz, then the Alberti representation can be chosen $\phi$-independent (i.e., in the $\phi$-direction of some cone $C$).
\end{corollary}

As a reminder, $\Curv(X)$ by definition does not contain any constant curves. Indeed, the purpose of this restriction is to prevent the following: A family of curves containing a constant curve would allow for no admissible functions and thus have modulus $\infty$. Such a family could not be used to construct a non-trivial Alberti representation, e.g., if it contained no non-constant curves. 

\begin{proof}

Fix $x_0\in X$. For each $n\in\mathbb{N}$, define 
$$\Gamma_{n} = \{\gamma \in \Gamma: \im(\gamma)\subseteq \overline{B}(x_0,n), {\rm diam}(\im(\gamma)) \geq \frac{1}{n}, 
{\rm LIP}(\gamma) \leq n, {\rm dom}(\gamma) \subseteq [-n,n]\}.$$
Since $\Gamma\subseteq \cup_{n=1}^\infty\Gamma_n$, \eqref{eq:subadditive} tells us that $\Mod_p(\Gamma_n,\mu)>0$ for some $n$, which we now fix. 

Consider the closure $\overline{\Gamma}_n \subset \Curv(X)$, which is
compact by Arzel\`a-Ascoli. By \eqref{eq:monotone},
$$ \Mod_p(\overline{\Gamma}_n,\mu) \geq \Mod_p(\Gamma_n,\mu)>0.$$
Moreover, since each curve $\gamma\in\overline{\Gamma}_n$ has $\diam(\im(\gamma))\geq \frac{1}{n}$ and is contained in $\overline{B}(x_0,n)$, we have
$$ \Mod_p(\overline{\Gamma}_n,\mu) \leq n^p\mu(B(x_0,n)) < \infty.$$

By Proposition \ref{prop:modmeasure} we obtain a measure $P$ on $\Curv(X)$, so that the corresponding measure $\eta_P$ is absolutely continuous with respect to $\mu$. Then, applying Proposition \ref{prop:decomposition} with cone $C = \bR^n \setminus \{0\}$, we obtain a non-trivial Alberti representation for $\mu$ that is $\phi$-independent, i.e., in the $\phi$-direction of $C$. 
\end{proof}

\section{Proofs of the corollaries}\label{sec:corollaries}

In this section, we prove all the corollaries of our main result stated in the introduction.

Before beginning the proofs, the following basic lemma allows us to reduce problems of bi-Lipschitz embedding to Theorem \ref{thm:mainthm}.
\begin{lemma}\label{lem:ARbilip}
Let $X$ be a metric space with a doubling measure $\mu_0$. Suppose that $\mu\ll\mu_0$ supports $k$ $\phi$-independent Alberti representations, for some $\phi\colon X\rightarrow \RR^m$.

Let $f\colon X\rightarrow Y$ be a bi-Lipschitz homeomorphism.

Then
\begin{enumerate}[(i)]
\item $\hat{\mu}_0 :=f_{*}(\mu_0)$ is a doubling measure supported on $Y$.
\item $\hat{\mu} :=f_{*}(\mu) \ll \nu_0$.
\item $\hat{\mu}$ supports $k$ $\phi\circ f^{-1}$-independent Alberti representations.
\end{enumerate}
\end{lemma}
\begin{proof}
The first two statements are immediate from the definitions.

For the third statement, let $\cA_i = (P^i, \nu^i)$ be independent Alberti representations for $\mu$, for $i=1,\dots,k$.

Note that $f\colon X \rightarrow Y$ and $f^{-1}\colon Y \rightarrow X$ induce continuous maps
$$ F: \Frag(X) \rightarrow \Frag(Y) \text{ and } F^{-1}\colon \Frag(Y)\rightarrow \Frag(X),$$
by post-composition. 

Let $\hat{\nu}\colon \Frag(Y) \rightarrow M(Y)$ be defined by
$$ \hat{\nu}_\gamma = f_{*}(\nu_{F^{-1}(\gamma)}) \text{ for each } \gamma\in\Frag(Y).$$
It is immediate that 
$$  \hat{\nu}_\gamma \ll \HH^1|_{\im(\gamma)}  \text{ for each } \gamma\in\Frag(Y),$$
since bi-Lipschitz maps preserve sets of zero $\HH^1$-measure.

We therefore define the Alberti representations
$$ \hat{\cA}_i = (F_{*}(P^i), f_*{\nu^i})$$
for $i=1,\dots,k$.

It is easy to check that these satisfy conditions (i), (ii), and (iv) of Definition 2.6. For condition (iii), observe that if $A\subseteq Y$ is Borel and $i\in\{1,\dots,k\}$, then

\begin{align*}
\hat{\mu}(A) &= \mu(f^{-1}(A))\\
&= \int_{\Frag(X)} \nu_\gamma^i(f^{-1}(A)) dP^i(\gamma)\\
&= \int_{\Frag(X)} \nu_{F(\gamma)}^i(A) dP^i(\gamma)\\
&= \int_{\Frag(Y)} \hat{\nu}_\alpha^i(A) dP^i(\alpha),
\end{align*}
as desired.

Lastly, we check the independence of the new Alberti representations on $Y$. Let $\cA=(P,\nu)$ denote any one of the $k$ original Alberti representations $\cA_i$ above. Then there is a cone $C\subseteq \RR^k$ such that $(\phi\circ \gamma)'(t)\in C$ for $P$-a.e. $\gamma\in \Frag(X)$ and a.e. $t\in\dom(\gamma)$. Let $G\subseteq \Frag(X)$ be the full $P$-measure set on which this holds.

Let $\hat{G}=F(G)\subseteq \Frag(Y)$, a set of full $\hat{P}$-measure in $\Frag(Y)$. Consider any $\alpha \in \hat{G}\subseteq \Frag(Y)$. Then the fragment $\gamma$ defined by
$$ t\mapsto f^{-1}(\alpha(t))$$
is in $G$. Therefore, for a.e. $t\in \dom(\alpha)=\dom(\gamma)$,
$$ (\phi\circ f^{-1} \circ \alpha)'(t) = (\phi \circ \gamma)'(t)$$
is in $C$.

Thus, each new Albert representation $\hat{\cA}_i$ is in the $(\phi\circ f^{-1})$-direction of the same cone of which $\cA_i$ was in the $\phi$-direction. Thus, the representations $\hat{\cA}_i$ are $(\phi\circ f^{-1})$-independent.
\end{proof}

As a consequence, we can now prove Corollary \ref{cor:bilip}.

\begin{proof}[Proof of Corollary \ref{cor:bilip}]
Let $X$ be a complete metric space admitting a doubling Radon measure $\mu_0$. Suppose that a measure $\mu\ll\mu_0$ supports $k$ independent Alberti representations, for some $k\geq 1$.

Suppose that $f\colon X \rightarrow \RR^n$ is a bi-Lipschitz embedding. Let $X'=f(X)$, $\mu'_0= f_{*}(\mu_0)$, and $\mu'=f_{*}(\mu)$. 

It follows from Lemma \ref{lem:ARbilip} that $\mu'$ admits $k$ independent Alberti representations. Therefore at $\mu'$-a.e. point $x'\in X'$, every tangent $Y'\in\Tan_{\RR^n}(X',x')$ is isometric to $Z\times \RR^k$, for some closed set $Z\subseteq \RR^{n-k}$.

The set of all preimages under $f$ of such points $x'\in X'$ forms a set of full $\mu$-measure in $X$. At such a point $x=f^{-1}(x')\in X$, each tangent $(Y,y)\in \Tan(X,x)$ is bi-Lipschitz equivalent to an element of $\Tan_{\RR^n}(X',x')$, with a bi-Lipschitz map given by a tangent map of $f$. Thus, any such tangent $Y$  is bi-Lipschitz equivalent to a product $Z\times \RR^k$, for some complete metric space $Z$. This proves the corollary.
\end{proof}

\subsection{Modulus and conformal dimension}

Here we prove Corollaries \ref{cor:positivemodulus}, \ref{cor:cdim}, and \ref{cor:carpet}.

\begin{proof}[Proof of Corollary \ref{cor:positivemodulus} ]

Let $X$ be a complete metric space admitting a Radon measure $\mu$ that is absolutely continuous with respect to a doubling measure $\mu_0$. Suppose that $X$ contains a family of (non-constant) curves $\Gamma$ so that $\Mod_p(\Gamma,\mu)>0$ for some $p \in [1,\infty)$, and that $X$ admits a bi-Lipschitz embedding $\phi$ into some $\RR^n$.

By Corollary \ref{cor:modmeasure}, there is a non-trivial measure on $X$ that is absolutely continuous to $\mu$, hence to $\mu_0$, and supports a $\phi$-independent Alberti representation. The corollary then follows from Corollary \ref{cor:bilip}.

\end{proof}

\begin{proof}[Proof of Corollary \ref{cor:cdim}]
Let $X$ satisfy the assumptions of the corollary. Thus, $X$ is Ahlfors $Q$-regular with $Q=\cdim(X)$ and $X$ admits a bi-Lipschitz embedding into some $\RR^n$. 

By the Keith-Laakso Theorem \ref{thm:KL}, there is a weak tangent $W$ of $X$ that contains a family of non-constant curves with positive modulus. By Lemma \ref{lem:tangentprops}, the space $W$ also admits a bi-Lipschitz embedding into $\RR^n$. By Corollary \ref{cor:positivemodulus}, there is a tangent $Y$ of $W$ that is bi-Lipschitz equivalent to $Z\times \RR$ for some complete metric space $Z$. 

As $Y$ is also a weak tangent of the original space $X$ (see Lemma \ref{lem:tangentprops}(vii)), this completes the proof.
\end{proof}

\begin{proof}[Proof of Corollary \ref{cor:carpet}]
Let $X$ be linearly connected and let $Y$ be quasisymmetric to $X$ and Ahlfors $Q$-regular, where 
$$Q=\cdim(X)=\cdim(Y)\in (1,2).$$
Suppose that $Y$ did admit a bi-Lipschitz embedding into some Euclidean space, $\RR^N$. 

By Corollary \ref{cor:cdim}, $Y$ would then admit a weak tangent $W$ that is bi-Lipschitz equivalent to a product $Z\times \RR$, for some complete metric space $Z$. Let $\phi\colon Z\times \RR \rightarrow W$ be bi-Lipschitz. Let $\pi\colon Z\times \RR\rightarrow Z$ be the projection to the $Z$ factor.

The linear connectedness condition is preserved under both quasisymmetry and passage to weak tangents. (The former assertion is immediate from the definitions, and the latter is contained in the proof of \cite[Proposition 5.4]{Kin}.) Thus, $W$ is linearly connected. By Lemma \ref{lem:tangentprops}(v), $W$ is also Ahlfors $Q$-regular.

Next, we observe that $Z$ must contain at least two points: if not, then $W$ would be bi-Lipschitz equivalent to $Z\times \RR \cong \RR$, which would contradict the fact that $W$ is Ahlfors $Q$-regular for $Q>1$.

We now observe that $Z$ must contain a compact, connected set $K$ with at least two points. To see this, fix distinct points $z, z'\in Z$. Since $W$ is linearly connected, there is a compact, connected set $J$ in $W$ that contains $\phi(z,0)$ and $\phi(z',0)$. The set $K=\pi(\phi^{-1}(J))$ is then a continuum containing $z$ and $z'$.

Thus, $Z$ contains a non-trivial continuum $K$ and so $W$ contains a bi-Lipschitz image of the space $K\times [0,1]$. We now argue that
\begin{equation}\label{eq:hausdim}
\dim_H(K\times [0,1])\geq 2.
\end{equation}
That the Hausdorff dimension of a product is at least the sum of the dimensions of the factors is standard for compact subsets of Euclidean space (see, e.g., \cite[Theorem 3.2.1]{BP}); we give a brief argument in our setting here:

Since $K$ is compact and connected, $\HH^1(K)>0$. By Frostman's Lemma \cite[Theorem 8.17]{Mat}, $K$ supports a Radon measure $\mu$ satisfying 
$$ \mu(B(x,r))\leq r \text{ for all }  x\in K \text{ and } 0<r\leq \diam(K).$$
If $\mathcal{L}^1$ denotes Lebesgue measure on $[0,1]$, then the measure $\mu \times \mathcal{L}^1$ on $K\times [0,1]$ satisfies
$$ (\mu\times\mathcal{L}^1)(B(p,r)) \leq r^2 \text{ for all }  p\in K\times [0,1] \text{ and } 0<r\leq \diam(K\times [0,1]).$$
By the ``mass distribution principle'' (see \cite[p. 61]{He}), we obtain \eqref{eq:hausdim}.

We therefore arrive at
$$ 2 > Q = \dim_{H}(W) \geq \dim_{H}(K\times [0,1]) \geq 2,$$
a contradiction.

\end{proof}

\subsection{Slit carpet}

Here, we prove Corollary \ref{cor:slitcarpet}. This will follow from Corollary \ref{cor:positivemodulus} and some facts about the slit carpet.

We first summarize some results of Merenkov \cite{Me}.
\begin{proposition}[Lemma 2.1, Proposition 2.4, and Lemma 4.2 of \cite{Me}]\label{prop:merenkov}
The slit carpet $\mathbb{M}$ has the following properties:
\begin{enumerate}[(i)]
\item It is homeomorphic to the standard Sierpi\'nski carpet. In particular, it is compact and has topological dimension $1$.\footnote{Here, ``topological dimension'' can refer to Lebesgue covering dimension or (small) inductive dimension, which are equivalent for compact metric spaces \cite{Nagata}. All we will need to know is that $\mathbb{M}$ does not contain a topologically embedded copy of any open subset of $\RR^2$.}
\item It is geodesic and Ahlfors $2$-regular.
\item It admits a family of non-constant curves with positive $2$-modulus (with respect to the measure $\HH^2$).
\end{enumerate}
\end{proposition}

\begin{proof}[Proof of Corollary \ref{cor:slitcarpet}]
By Proposition \ref{prop:merenkov}, we may fix a point $p\in \M$ where the conclusion of Corollary \ref{cor:positivemodulus} holds.

The self-similarity of $\M$ easily implies the following: There is a constant $c>0$ such that, for each $r>0$, there is a point $q_r$ and a compact set $K_r \subseteq B(p,r)$ such that
$$ B(q_r,cr) \subseteq K_r \subseteq B(p,r),$$
and
$$ K_r \text{ is isometric to } t\M \text{ for some } t \in (2cr,2r).$$

Let $(Y,y)\in\Tan(\M,p)$ be obtained by rescaling along a sequence $\lambda_i\rightarrow 0$. Let $q_i = q_{\lambda_i}$ and $K_i = K_{\lambda_i}$. Passing to a (subsequential) limit, the above properties imply (see, e.g., \cite[Lemma 8.31]{DS}) that there is a point $q\in Y$ and a compact set $K\subseteq Y$ such that
$$ B(q,c) \subseteq K \subseteq \overline{B}(y, 1)$$
and
$$ K \text{ is bi-Lipschitz equivalent to } \M.$$

By Corollary \ref{cor:positivemodulus}, $Y$ is bi-Lipschitz equivalent to $Z\times \RR$ for some complete (and necessarily doubling) metric space $Z$. As a tangent of $\M$, $Y$ is quasiconvex and therefore so is $Z$.

It follows that $Z$ contains a non-trivial topological arc through each point. Hence, there is a homeomorphic image of $[0,1]^2$ contained in $B(q,c)$. This implies that $B(q,c)\subseteq K$ must have topological dimension $2$. However, this contradicts the fact that $K$ is bi-Lipschitz equivalent to $\M$, which has topological dimension $1$.

\end{proof}

\subsection{Heisenberg group}\label{subsec:heisenberg}

Here we give a brief introduction to the Heisenberg group and prove Corollary \ref{cor:Heisenberg}.

\subsubsection{Preliminaries on the Heisenberg group}
We now fix some notation and definitions. We will be very brief, referring the reader to \cite{LDsurvey, CDPT} for details.

The group $\bH$ is $\RR^3$ endowed with the non-abelian group law
$$ (x_1, y_1, z_1)\cdot (x_2, y_2, z_2) = (x_1+x_2, y_1+y_2, z_1+z_2 + \frac{1}{2}(x_1 y_2 - y_1 x_2)).$$
There are many standard, bi-Lipschitz equivalent ways to equip $\bH$ with a metric. For simplicity, we fix the so-called Kor\'anyi distance, though it will make little difference below.

\begin{definition}
The \textit{Kor\'anyi norm} of $(x,y,z)\in\bH$ is
$$\|(x,y,z)\| = ((x^2+y^2)^2+16z^2)^{1/4}.$$
The \textit{Kor\'anyi distance} between $p, q\in\bH$ is
$$d(p,q) = \| p^{-1}q\|.$$
\end{definition}
The Kor\'anyi distance on $\bH$ induces the usual topology from $\RR^3$ and has the following features (see, e.g. \cite[Example 1.3]{LDsurvey}):
\begin{enumerate}[(i)]
\item Left-invariance: $d(p,q) = d(p'\cdot p, p'\cdot q)$ for all $p,q,p'\in\bH$.
\item Dilations: For each $t>0$, the map 
$$ \delta_t(x,y,z) = (tx,tx, t^2z)$$
is a group homomorphism with the property that
$$d(\delta_t(p),\delta_t(q))=td(p,q) \text{ for all } p,q\in\bH$$
\item Doubling: The Lebesgue measure $\cL$ on $\RR^3$ is doubling on $(\bH,d)$. In fact, it satisfies the Ahlfors $4$-regularity property
$$ \cL(B(p,r)) \approx r^4$$
for all $p\in\bH$, $r>0$ and a fixed positive implied constant.
\end{enumerate}
In particular, item (iii) implies that every open set in $\bH$ has Hausdorff dimension $4$.

The Kora\'nyi distance is also bi-Lipschitz eqiuvalent to the Carnot-Carath\'eodory distance, which we do not define here, as both satisfy properties (i)-(iii) above \cite{LDsurvey}.

From these properties, we first derive the following:
\begin{lemma}\label{lem:heisenbergAR}
There is an open set $U$ in the Heisenberg group such that $\cL|_U$ supports two independent Alberti representations.
\end{lemma}
This fact is well-known, and of course much more is true. We include a brief proof only to show that no sophisticated tools are needed.

\begin{proof}
Let $\phi\colon \bH \to \RR^2$ be the canonical Lipschitz chart, given by $(x,y,z)\mapsto (x,y)$. Fix any disjoint, independent cones, $C_x,C_y$ in $\bR^2$ containing the $x$ and $y$-axis respectively.

For each $p=(a,b)\in\RR^2$, the curves
$$\alpha_p(t)= \left(t,a,b-\frac{1}{2}at\right)$$
and
$$\beta_p(t) = \left(a,t,b+\frac{1}{2}at\right)$$
are bi-infinite geodesics in $\bH$. Indeed, 
$$\alpha_{(a,b)}(t)= (0,a,b)\cdot (t,0,0) \text{ and } \beta_{(a,b)}(t)= (a,0,b)\cdot (0,t,0),$$
and the curves $(t,0,0)$ and $(0,t,0)$ are clearly geodesics.

Define the maps $\Phi_\alpha$ and $\Phi_\beta$ from $\left[-\frac{1}{2},\frac{1}{2}\right]^3$ to $\RR^3$ by $(p,t) \mapsto \alpha_p(t)$ and $(p,t) \mapsto \beta_p(t)$, respectively. One directly verifies that these maps are injective and are open mappings on the interiors of their domains. Since $\Phi_\alpha(0)=\Phi_\beta(0)=(0,0,0)$, there is a non-empty open set $U$ contained in 
$$\Phi_\alpha\left(\left[-\frac{1}{2},\frac{1}{2}\right]^3\right) \cap \Phi_\beta\left(\left[-\frac{1}{2},\frac{1}{2}\right]^3\right).$$
We will obtain two independent Alberti representations of the measure $\mu = \cL|_U$ on $\bH$.

Computing the Jacobians gives immediately that $\Phi_\alpha$ and $\Phi_\beta$ are volume preserving. Writing $\lambda$ for Lebesgue measure on $\left[-\frac{1}{2},\frac{1}{2}\right]^2$, we therefore obtain by change of variables that 
\begin{equation}\label{eq:alpha}
\mu(A) = \int_{\bR^2} \int_\bR 1_A(\alpha_p(t)) ~dt ~d\lambda(p)
\end{equation}
and
\begin{equation}\label{eq:beta}
\mu(A) = \int_{\bR^2} \int_\bR 1_A(\beta_p(t)) ~dt ~d\lambda(p)
\end{equation}
for any Borel set $A\subseteq U$.

Thus, define a probabilty measure $P$ on $\Frag(\bH)$ by the pushforward of $\lambda|_{\left[-\frac{1}{2},\frac{1}{2}\right]^2}$ under the map from $\left[-\frac{1}{2},\frac{1}{2}\right]^2$ given by
$$ p \mapsto \alpha_p|_{\left[-\frac{1}{2},\frac{1}{2}\right]} \in \Frag(\bH).$$
Note that this map is continuous on $\left[-\frac{1}{2},\frac{1}{2}\right]^2$.

For $\gamma\in \Frag(X)$, define $\nu_\gamma = 0$ if $\gamma$ is not in the support of $P$. Otherwise, $\gamma= \alpha_p|_{\left[-\frac{1}{2},\frac{1}{2}\right]} $ for some $p$, and we set
$$\nu_\gamma = 1_U\cdot\HH^1|_{\im(\gamma)}.$$
The pair $(P,\nu)$ defines an Alberti representation of $\mu$ by \eqref{eq:alpha}, supported on curves of the form $\alpha_p|_{\left[-\frac{1}{2},\frac{1}{2}\right]}$. The exact same procedure applied to the curves $\beta_p$ yields an Alberti representation of $\mu$ supported on curves of the form $\beta_p|_{\left[-\frac{1}{2},\frac{1}{2}\right]}$. 

Since $(\phi \circ \alpha_p)'(t) = (1,0)\in C_x$ and $(\phi \circ \beta_p)'(t)=(0,1)\in C_y$ for each $p\in \RR^2$ and $t\in \RR$, the two Alberti representations are independent.

\end{proof}

\subsubsection{Proof of Corollary \ref{cor:Heisenberg}}

There are a number of proofs of Corollary \ref{cor:Heisenberg} in the literature. We give a proof below that avoids Pansu's differentiation theorem from \cite{Pansu}. It relies only on Theorem \ref{thm:mainthm}, invariance of domain, and the basic properties of the Heisenberg group stated in the previous subsection. Of course, we still use a blowup argument, so the ideas are similar in spirit.

\begin{proof}[Proof of Corollary \ref{cor:Heisenberg}]
Suppose that the Heisenberg group (with the Kor\'anyi metric) admitted a bi-Lipschitz embedding into some Euclidean space. By Lemma \ref{lem:heisenbergAR} and Corollary \ref{cor:bilip}, some tangent of the Heisenberg group would be bi-Lipschitz equivalent to $Z\times \RR^2$, for some complete metric space $Z$.

By the homogeneity and dilation structure of the metric, every tangent of the Heisenberg group is isometric to the Heisenberg group itself. Thus, in this case, $\bH$ is bi-Lipschitz equivalent to $Z\times \RR^2$. Since $\bH$ is proper and quasiconvex, so is $Z$. 

It follows that $Z$ is bi-Lipschitz equivalent to a geodesic metric space, simply by replacing the metric on $Z$ by the associated length metric. Therefore, in particular, $Z$ contains a Lipschitz embedding $\gamma\colon [0,1]\rightarrow Z$.

Let $\phi\colon Z\times\RR^2 \rightarrow \bH$ be bi-Lipschitz. The map from $(0,1)\times (0,1)^2$ into $\bH$ given by
$$ (t,p) \mapsto \phi(\gamma(t),p)$$
is therefore a Lipschitz homeomorphism from an open set of $\RR^3$ into $\bH$. By invariance of domain, the image of this map must be open in $\bH$. On the other hand, this set has Hausdorff dimension at most $3$, as the Lipschitz image of a subset of $\RR^3$. This violates the Ahlfors $4$-regularity of $\bH$.
\end{proof}

\section{Appendix: Proof of Proposition \ref{prop:basepoint}}\label{sec:appendix}
Here we give a proof of Proposition \ref{prop:basepoint}. The idea is extremely similar to that of \cite[Proposition 3.1]{GCD} (which in turn is based on \cite{Preiss, LD}), which is the same statement in the setting of Gromov-Hausdorff tangents rather than intrinsic tangents. We will therefore omit many steps if they are easy to adapt from there.

Recall the notion of the \textit{outer measure}
$$\mu^*(A)=\inf\{\mu(B): B\text{ Borel }, B\supseteq A\},$$
and the associated notion of a \textit{point of outer density} $x$ of a set $A$, where
$$ \lim_{r\rightarrow 0} \frac{\mu^*(B(x,r) \cap A)}{\mu(B(x,r))} \rightarrow 1.$$
As explained briefly in \cite{GCD}, every set of positive outer measure has a point of outer density.

The following lemma is a simple extension of (vi) of Lemma \ref{lem:tangentprops}. See also \cite[Lemma 9.6]{DS} or \cite[Proposition 3.1]{LD} for closely related statements whose proofs can easily be modified to yield this one.
\begin{lemma}\label{lem:density}
Let $A\subseteq \RR^n$ support a doubling measure $\mu$ and a Lipschitz $f\co A \rightarrow \RR^m$. Let $E\subseteq A$ have a point of outer $\mu$-density at $a\in A$. Then
$$ \Tan_{\RR^n}(E,f,a) = \Tan_{\RR^n}(A,f,a).$$
\end{lemma}

We now define a notion of distance that yields the correct topology. As in \cite{GCD}, the distance we define will not precisely be a metric, but it will suffice for our purposes.
\begin{definition}
Let $A, B \subseteq \RR^N$ be sets and $f,g\colon \RR^N\rightarrow \RR^M$, Lipschitz functions. 
Define
$$ \tilde{D}((A,f),(B,g)) = \inf\{ \epsilon>0 : d_{1/\epsilon}(A,B)<\epsilon \text{ and } |f-g|<\epsilon \text{ on } (A\cup B) \cap B(0,1/\epsilon)\}.$$

Then define
$$ D = \min(\tilde{D},1/2).$$
\end{definition}

To simplify notation, given $A\subseteq\RR^N$, $f\co \RR^N \rightarrow \RR^M$ Lipschitz, $\lambda>0$, and $p\in\RR^N$, we set
$$ A_{p,\lambda} = \lambda^{-1}(A-p)$$
and
$$ f_{p,\lambda}(x) = \lambda^{-1}(f(\lambda x + p) - f(p)).$$
Note that $f_{p,\lambda}$ is Lipschitz with the same constant as $f$.

The following is analogous to \cite[Lemma 2.3]{GCD}.
\begin{lemma}\label{lem:Dprops}
The function $D$ has the following properties:
\begin{enumerate}[(i)]
\item It is non-negative and symmetric.
\item If $D((A,f),(B,g))=0$ then $\overline{A}=\overline{B}$ and $f=g$ on $\overline{A}=\overline{B}$.
\item For all pairs $(A,f), (B,g), (C,h)$, we have the quasi-triangle inequality 
$$ D((A,f),(C,h)) \leq 2(D((A,f),(B,g)) + D((B,g),(C,h)))$$
\item $(\hat{A}, \hat{f})\in\Tan_{\RR^N}(A,f,a)$ if and only if some Lipschitz extensions of $f$ and $\hat{f}$ to all of $\RR^N$ and some sequence $\lambda_i\rightarrow 0$ satisfy
\begin{equation}\label{eq:newmetric}
D( (A_{a,\lambda_i}, f_{a,\lambda_i}), (\hat{A},\hat{f})) \rightarrow 0
\end{equation}
\end{enumerate}
\end{lemma}
\begin{proof}
The first two items are simple, the third follows exactly as in \cite[Lemma 2.3]{GCD}, and the fourth follows from \cite[Lemma 8.7]{DS}. 
\end{proof}

The next lemma is an analog of \cite[Lemma 2.6]{GCD}.

\begin{lemma}\label{lem:separable}
For each $N,M\in\mathbb{N}$ and $L,\eta>0$, the collection
$$ \mathcal{S}= \{ (B,g) : B\subseteq \RR^N, g\co \RR^N\rightarrow\RR^M L\text{-Lipschitz}\}$$
is contained in a countable collection of sets $B_\ell$ with $D$-diameter at most $\eta$.
\end{lemma}
The ``$D$-diameter'' of a collection of pairs $\{(B,g)\}$ is the supremum of the $D$-distance between pairs of elements in the collection.
\begin{proof}
Consider all pairs $(K,h)$ such that $K\subseteq \mathbb{Q}^N\subseteq \RR^{N}$ is finite, and $h\colon K \rightarrow \mathbb{Q}^M$.

By Lemma \ref{lem:Dprops}(iii), it suffices to show, given $\eta\in (0,1)$ and $(B,g)\in\mathcal{S}$, that $(B,g)$ is within $D$-distance $10\eta$ of some $(K,\hat{h})$, where $(K,h)$ is as above and $\hat{h}$ is a Lipschitz extension of $h$ to all $\RR^N$. We may also assume $L> 1$.

Fix $K\subseteq B(0,2\eta^{-1}) \cap \mathbb{Q}^N$ to be $\eta/L$-separated and finite such that 
$$ d_{2\eta^{-1}}(K, B) \leq \eta/L.$$
Then, for $x\in K$, set $h(x)$ to be an element of $\mathbb{Q}^M$ within distance $\eta/L$ of $g(x)$. Note that $h$ is $3L$-Lipschitz. Extend $h$ to a $3L$-Lipschitz map $\hat{h}$ on $\RR^N$ by Kirszbraun's theorem. For $x\in K \cap B(0,\eta^{-1})$, 
$$ |g(x) - \hat{h}(x)| \leq \eta/L < \eta$$
and if $x\in B \cap B(0,\eta^{-1})$, then
$$ |g(x)-\hat{h}(x)| \leq 6\eta + |g(y)-h(y)| \leq 10\eta,$$
where $y$ is a closest element in $K$ to $x$. This completes the proof.
\end{proof}

\begin{proof}[Proof of Proposition \ref{prop:basepoint}]
We closely follow the argument in \cite{GCD}. Let $A \subseteq \RR^n$ support a doubling measure $\mu$. Let $f\colon A \rightarrow \RR^m$ be a Lipschitz mapping. Extend $f$ to be an $L$-Lipschitz function defined on all of $\RR^N$ by the standard McShane-Whitney extension theorem.

Our goal is to show that the set
$$ \left\{a\in A: \text{ there exists } (B,g)\in \Tan_{\RR^N}(A,f) \text{ and } b\in B \text{ such that } (B-b,g(\cdot+b)-g(b))\notin \Tan_{\RR^N}(A,f)\right\} $$
has outer measure zero.

Consider the collection
$$\mathcal{S}= \{ (B,g) : B\subseteq \RR^N, g\co \RR^N\rightarrow\RR^M \text{ } L\text{-Lipschitz}\}.$$
Note that every rescaling or tangent of $(A,f)$ lies in $\mathcal{S}$.

Fix $k\geq 1$. Apply Lemma \ref{lem:separable} to obtain countably many collections $B_l$ such that, for all $l$,
$$\text{diam}_D(B_l) < 1/4k$$
and $\mathcal{S} \subseteq \cup B_l$. 

It therefore suffices to show that, for all $k,l,m \in \mathbb{N}$, the set of ``bad'' points with parameters $k,l,m$, namely
\begin{align}
\bigg\{a\in A: &\text{ there exists } (B,g)\in \Tan_{\RR^n}(A,f) \text{ and } b\in B \text{ such that } \label{eq:bad}\\
& (B-b,g(\cdot+b)-g(b)) \in B_l \text{ and } \nonumber\\
& D\left(\left(B-b,g(\cdot+b)-g(b)\right), \left(A_{a,t}, f_{a,t}\right)\right)>\frac{1}{k} \text{ for all } t\in\left(0,1/m \right) \bigg\} \nonumber
\end{align}
has outer measure zero.

Suppose that, for some $k,l,m\in\mathbb{N}$, the set above has positive outer measure, and call it $A'\subseteq A$. Let $a$ be a point of outer density of $A'$. Then there exist $(B,g)\in\Tan_{\RR^N}(A,f,a)$ and $b\in B$ such that
$$ (B-b, g(\cdot+b)-g(b))\in B_l $$
and
$$ D\left( \left(B-b,g(\cdot+b)-g(b)\right), \left(\frac{1}{t}(A-a), \frac{1}{t}(f-f(a))\right)\right) >\frac{1}{k},$$
for all $t\in\left(0,1/m \right).$

Because $(B,g)\in\Tan_{\RR^N}(A,f,a)=\Tan_{\RR^N}(A',f,a)$ by Lemma \ref{lem:density}, there is a sequence $\lambda_i\rightarrow 0$ such that
\begin{equation}\label{eq:ei}
 \epsilon_i := D\left((B,g), (A'_{a,\lambda_i}, f_{a,\lambda_i}) \right) \rightarrow 0.
\end{equation}
In particular, we may choose $a_i\in A'$ such that
\begin{equation}\label{eq:ai}
 | b - \lambda_i^{-1}(a_i-a)|  < \epsilon_i,
\end{equation}
when $i$ is sufficiently large.

\begin{claim}
We have
$$ D \left( (B-b, g(\cdot + b) - g(b)), (A_{\lambda_i, a_i}, f_{\lambda_i, a_i}) \right) \rightarrow 0.$$

\end{claim}
\begin{proof}[Proof of Claim]
Observe that if $q-b\in (B-b) \cap B(0,(2\epsilon_i)^{-1})$, then $q\in B \cap B(0,\epsilon_i^{-1})$ and therefore there is $p\in A$ such that
$$ |q-\lambda_i^{-1}(p-a)|<\epsilon_i.$$
It follows that
$$ |(q-b)-\lambda_i^{-1}(p-a_i)| \leq  |q-\lambda_i^{-1}(p-a)| + |b-\lambda_i^{-1}(a_i-a)| < 2\epsilon_i.$$
A similar argument also shows, conversely, that if $\lambda_i^{-1}(p-a_i)$ is an arbitrary point in $\lambda_i^{-1}(A-a_i)\cap B(0,(2\epsilon_i)^{-1})$, then there is a point $q-b\in B-b$ such that
$$ |(q-b)-\lambda_i^{-1}(p-a_i)|  < 2\epsilon_i.$$
Together, these show that
$$ d_{(2\epsilon_i)^{-1}}(B-b,  A_{a_i,\lambda_i}) < 2\epsilon_i.$$
By \eqref{eq:ei}, the functions $g$ and $f_{a, \lambda_i}$ agree up to error $\epsilon_i$ on $(B \cup A_{a,\lambda_i}) \cap B(0,\epsilon_i^{-1})$.

Note that if $q-b\in  (B-b) \cap B(0,(2\epsilon_i)^{-1})$, then $q\in B \cap B(0,\epsilon_i^{-1})$. Similarly, if $\lambda_i^{-1}(p-a_i)\in A_{a_i,\lambda_i}\cap B(0,(2\epsilon_i)^{-1})$, then $\lambda_{i}^{-1}(p-a)\in A_{a,\lambda_i}\cap B(0,\epsilon^{-1}).$

Thus, a basic calculation using \eqref{eq:ai} and the triangle inequality, extremely similar to that on \cite[p. 565]{GCD}, yields that the functions
$$ g(\cdot) - g(b) \text{ and } f_{\lambda_i, a_i}$$
agree up to error $\lesssim \epsilon_i$ on 
$$ ( (B-b) \cup A_{a_i,\lambda_i}) \cap B(0,(2\epsilon_i)^{-1}),$$
where the implied constant depends only on the Lipschitz constants of $f$ and $g$. This proves the claim.
\end{proof}
With the claim proven, one reaches a contradiction exactly as on \cite[p. 566]{GCD}. Therefore, the set of ``bad points'' defined in \eqref{eq:bad} has outer measure zero, which completes the proof.
\end{proof}

\bibliography{modulussplitting}{}
\bibliographystyle{plain}

\end{document}